%% file: InertialDynamics.tex
\documentclass[reqno]{amsart}

\usepackage{amsmath}
\usepackage{amssymb}
\usepackage{amsfonts}
\usepackage{amsthm}

\usepackage[utf8]{inputenc}
\usepackage[T1]{fontenc}

\usepackage[mathcal]{euscript}

\usepackage[dvipsnames,svgnames]{xcolor}
\colorlet{MyBlue}{DodgerBlue!75!Black}

\usepackage{subfigure}

\usepackage{acronym}
\usepackage{latexsym}
\usepackage{paralist}
\usepackage{xspace}

\usepackage[square,numbers,sort&compress]{natbib}

\usepackage{hyperref}
\hypersetup{
colorlinks=true,
linktocpage=true,
pdfstartview=FitH,
breaklinks=true,
pdfpagemode=UseNone,
pageanchor=true,
pdfpagemode=UseOutlines,
plainpages=false,
bookmarksnumbered,
bookmarksopen=false,
bookmarksopenlevel=1,
hypertexnames=true,
pdfhighlight=/O,
urlcolor=Maroon,linkcolor=MyBlue!60!black,citecolor=DarkGreen!70!black, 
pdftitle={},
pdfauthor={},
pdfsubject={},
pdfkeywords={},
pdfcreator={pdfLaTeX},
pdfproducer={LaTeX with hyperref}
}

\newcommand{\R}{\mathbb{R}}

\DeclareMathOperator{\argdot}{\cdot}

\DeclareMathOperator{\bd}{bd}

\DeclareMathOperator{\diag}{diag}

\DeclareMathOperator{\exclude}{\backslash}

\DeclareMathOperator{\grad}{grad}

\DeclareMathOperator{\hess}{Hess}

\DeclareMathOperator{\supp}{supp}

\newcommand{\dd}{\,d}
\newcommand{\eps}{\varepsilon}
\newcommand{\from}{\colon}
\newcommand{\injects}{\hookrightarrow}

\newcommand{\pd}{\partial}
\newcommand{\simplex}{\Delta}

\newcommand{\abs}[1]{\left\lvert #1 \right\rvert}
\newcommand{\smallabs}[1]{\lvert #1 \rvert}
\newcommand{\norm}[1]{\left\| #1 \right\|}
\newcommand{\smallnorm}[1]{\| #1 \|}

\newcommand{\product}[2]{\left\langle #1,  #2 \right\rangle}
\newcommand{\smallproduct}[2]{\langle #1,  #2 \rangle}

\newcommand{\braket}[2]{\left\langle #1 \middle\vert  #2 \right\rangle}

\newcommand{\defeq}{\equiv}
\newcommand{\eqdef}{\equiv}

\newcommand{\txs}{\textstyle}

\newcommand{\insum}{\sum\nolimits}

\usepackage{datetime}
\usepackage[textwidth=35mm]{todonotes}
\usepackage{soul}
\setstcolor{red}
\sethlcolor{SkyBlue}

\theoremstyle{plain}
\newtheorem{theorem}{Theorem}
\newtheorem{corollary}[theorem]{Corollary}
\newtheorem*{corollary*}{Corollary}
\newtheorem{lemma}[theorem]{Lemma}
\newtheorem{proposition}[theorem]{Proposition}

\theoremstyle{definition}

\newtheorem*{definition*}{Definition}

\newtheorem{remark}{Remark}
\newtheorem{example}{Example}

\newenvironment{proofof}[1]{\begin{proof}[#1]}{\end{proof}}

\numberwithin{equation}{section}
\numberwithin{theorem}{section}
\numberwithin{remark}{section}
\numberwithin{example}{section}

\newcommand{\play}{\mathcal{N}}
\newcommand{\act}{\mathcal{A}}
\newcommand{\pay}{u}
\newcommand{\payv}{v}
\newcommand{\strat}{\mathcal{X}}
\newcommand{\game}{\mathfrak{G}}

\newcommand{\eq}{x^{\ast}}
\newcommand{\eqset}{\strat^{\ast}}
\newcommand{\omlim}{x^{\omega}}

\newcommand{\kin}{K}
\newcommand{\pot}{\Phi}
\newcommand{\temp}{\eta}

\usepackage{xstring}

\newcommand{\set}{\mathcal{S}}

\newcommand{\bvec}{e}

\newcommand{\orthant}{\R_{++}}
\newcommand{\orthalt}{C}
\newcommand{\intstrat}{\strat^{\circ}}
\newcommand{\intsimplex}{\simplex\!^{\circ}}
\newcommand{\open}{U}
\newcommand{\nhd}{U}

\newcommand{\del}{\nabla}
\newcommand{\vel}{\upsilon}
\newcommand{\Fsup}{F_{\textup{max}}}
\newcommand{\vsup}{\vel_{\textup{max}}}

\newcommand{\friction}{\eta}

\newcommand{\trans}{\top}

\newcommand{\sections}{\mathcal{T}}

\begin{document}


\title
[Inertial game dynamics]
{Inertial game dynamics and\\%
applications to constrained optimization}

\author[R.~Laraki]{Rida Laraki}
\address
[R.~Laraki]
{CNRS (French National Center for Scientific Research), LAMSADE\textendash Paris-Dauphine, Paris, France\\
and
École Polytechnique, Department of Economics, Paris, France}
\email{\href{mailto:rida.laraki@lamsade.dauphine.fr}{rida.laraki@lamsade.dauphine.fr}}
\urladdr{\url{https://sites.google.com/site/ridalaraki}}

\author[P.~Mertikopoulos]{Panayotis Mertikopoulos}
\address
[P.~Mertikopoulos]
{CNRS (French National Center for Scientific Research), LIG, F-38000 Grenoble, France\\
and
Univ. Grenoble Alpes, LIG, F-38000 Grenoble, France}
\email{\href{mailto:panayotis.mertikopoulos@imag.fr}{panayotis.mertikopoulos@imag.fr}}
\urladdr{\url{http://mescal.imag.fr/membres/panayotis.mertikopoulos}}

\thanks{The authors are sincerely grateful to Jérôme Bolte for proposing the term ``inertial'' and for many insightful discussions.}
\thanks{The authors gratefully acknowledge financial support from the French National Agency for Research under grants ANR-10-BLAN-0112-JEUDY, ANR-13-JS01-GAGA-0004-01 and ANR-11-IDEX-0003-02/Labex ECODEC ANR-11-LABEX-0047 (part of the program ``Investissements d'Avenir'').
The second author was also partially supported by the Pôle de Recherche en Mathématiques, Sciences, et Technologies de l'Information et de la Communication under grant no. C-UJF-LACODS MSTIC 2012 and the European Commission in the framework of the FP7 Network of Excellence in Wireless COMmunications NEWCOM\# (contract no. 318306).}

\newacro{ESS}{evolutionarily stable state}
\newacro{EW}{exponential weights}
\newacro{HR}{Hess\-i\-an\textendash Rie\-man\-ni\-an}
\newacro{QR}{quantal response}
\newacro{KKT}{Ka\-rush\textendash Kuhn\textendash Tuc\-ker}
\newacro{RPS}[\textsf{RPS}]{Rock-Paper-Scissors}
\newacro{MP}{Matching Pennies}
\newacro{ODE}{ordinary differential equation}
\newacro{OD}[O/D]{origin\textendash destination}
\newacro{OMD}{online mirror descent}
\newacro{FTRL}{follow the regularized leader}
\newacro{MD}{mirror descent}
\newacro{KL}{Kullback\textendash Leibler}



\begin{abstract}
\input{Abstract}
\end{abstract}

\maketitle

\acresetall


\vspace{-3em}
\setcounter{tocdepth}{1}
\tableofcontents

\section{Introduction}
\label{sec:introduction}
\input{Introduction}

\section{Inertial game dynamics}
\label{sec:dynamics}
\input{Dynamics}

\section{Basic properties and well-posedness}
\label{sec:asymptotics}
\input{Asymptotics}

\section{Long-term optimization and rationality properties}
\label{sec:results}
\input{Results}

\section{Discussion}
\label{sec:discussion}
\input{Discussion}
\appendix

\section{Elements of Riemannian geometry}
\label{app:geometry}
\input{App-Geometry}

\section{Calculations and proofs}
\label{app:calculations}
\input{App-Calculations}

\bibliographystyle{siam}
\bibliography{Bibliography}

\end{document}

%% file: Abstract.tex
Aiming to provide a new class of game dynamics with good long-term rationality properties, we derive a second-order inertial system that builds on the widely studied ``heavy ball with friction'' optimization method.
By exploiting a well-known link between the replicator dynamics and the Shahshahani geometry on the space of mixed strategies, the dynamics are stated in a Riemannian geometric framework where trajectories are accelerated by the players' unilateral payoff gradients and they slow down near Nash equilibria.
Surprisingly (and in stark contrast to another second-order variant of the replicator dynamics), the inertial replicator dynamics are not well-posed;
on the other hand, it is possible to obtain a well-posed system by endowing the mixed strategy space with a different \ac{HR} metric structure and we characterize those \ac{HR} geometries that do so.
In the single-agent version of the dynamics (corresponding to constrained optimization over simplex-like objects), we show that regular maximum points of smooth functions attract all nearby solution orbits with low initial speed.
More generally, we establish an inertial variant of the so-called ``folk theorem'' of evolutionary game theory and we show that strict equilibria are attracting in asymmetric (multi-population) games \textendash\ provided of course that the dynamics are well-posed.
A similar asymptotic stability result is obtained for \aclp{ESS} in symmetric (single-population) games.

%% file: Introduction.tex

One of the most widely studied dynamics for learning and evolution in games is the classical replicator equation of Taylor and Jonker \cite{TJ78}, first introduced as a model of population evolution under natural selection.
Stated in the context of finite $N$-player games with each player $k\in\{1,\dotsc,N\}$ choosing an action from a finite set $\act_{k}$, these dynamics take the form:
\begin{equation}
\label{eq:RD}
\tag{RD}
\dot x_{k\alpha}
	= x_{k\alpha} \left[ \payv_{k\alpha}(x) - \insum_{\beta\in\act_{k}} x_{k\beta} \payv_{k\beta}(x) \right],
\end{equation}
where $x_{k} = (x_{k\alpha})_{\alpha\in\act_{k}}$ denotes the mixed strategy of player $k$ (i.e. $x_{k\alpha}$ represents the probability with which player $k$ selects $\alpha\in\act_{k}$) while $\payv_{k\alpha}(x)$ denotes the expected payoff to action $\alpha\in\act_{k}$ in the mixed strategy profile $x = (x_{1},\dotsc,x_{N})$.%
\footnote{In the mass-action interpretation of population games, $x_{k\alpha}$ represents the proportion of players in population $k$ that use strategy $\alpha\in\act_{k}$ and $\payv_{k\alpha}(x)$ is the associated fitness.}

Accordingly, a considerable part of the literature has focused on the long-term rationality properties of the replicator dynamics.
First, building on early work by Akin \cite{Aki80} and Nachbar \cite{Nac90}, Samuelson and Zhang \cite{SZ92} showed that dominated strategies become extinct along every interior trajectory of \eqref{eq:RD}.
Second, the so-called ``folk theorem'' of evolutionary game theory states that
\begin{inparaenum}
[\itshape a\upshape)]
\item
Nash equilibria are stationary in \eqref{eq:RD};
\item
limit points of interior trajectories are Nash;
and
\item
strict Nash equilibria are asymptotically stable under \eqref{eq:RD} \cite{HS98,HS03}.
\end{inparaenum}
Finally, when the game admits a potential function (in the sense of \cite{MS96}), interior trajectories of \eqref{eq:RD} converge to the set of Nash equilibria that are local maximizers of the game's potential \cite{HS98}.

To a large extent, the strong rationality properties of the replicator dynamics are owed to their dual nature as a reinforcement learning/unilateral optimization device.
The former aspect is provided by the link between \eqref{eq:RD} and the so-called \ac{EW} algorithm where players
choose an action with probability that is exponentially proportional to its cumulative payoff over time \cite{LW94,MM10,Rus99,Sor09,Vov90}.
In continuous time, this process formally amounts to the dynamical system:
\begin{equation}
\label{eq:EW}
\tag{EW}
\begin{aligned}
\dot y_{k\alpha}
	&= \payv_{k\alpha},
	\\
x_{k\alpha}
	&= \frac{\exp(y_{k\alpha})}{\sum_{\beta} \exp(y_{k\beta})},
\end{aligned}
\end{equation}
and, as can be seen by a simple differentiation, \eqref{eq:EW} is equivalent to \eqref{eq:RD}.

Dually, from an optimization perspective, the replicator dynamics can also be seen as a unilateral gradient ascent scheme where, to maximize their individual payoffs, players ascend the (unilateral) gradient of their payoff functions with respect to a particular geometry on the simplex \textendash\ the so-called \emph{Shahshahani metric}, given by the metric tensor $g_{\alpha\beta}(x) = \delta_{\alpha\beta}/x_{\alpha}$ for $x_{\alpha}>0$ \cite{Sha79}.
In this light, \eqref{eq:RD} can be recast as:
\begin{equation}
\label{eq:GD-S}
\dot x_{k}
	= \grad_{k}^{S} \pay_{k}(x),
\end{equation}
where $\grad_{k}^{S}\pay_{k}(x)$ denotes the unilateral Shahshahani gradient of the expected payoff function $\pay_{k}(x) = \sum_{\alpha} x_{k\alpha} \payv_{k\alpha}(x)$ of player $k$ \cite{Aki79,Hof96,HS90,Sha79}.%
\footnote{For our purposes, ``unilateral'' here means differentiation with respect to the variables that are directly under the player's control (as opposed to all variables, including other players' strategies).}
Owing to this last interpretation, \eqref{eq:RD} becomes a proper Shahshahani gradient ascent scheme in the class of potential games:
the game's potential acts as a global Lyapunov function for \eqref{eq:RD}, so interior trajectories converge to the set of Nash equilibria that are local maximizers thereof \cite{HS90,HS98}.%
\footnote{By contrast, using ordinary Euclidean gradients and projections leads to the well-known (Euclidean) projection dynamics of Friedman \cite{Fri91};
however, because Euclidean trajectories may collide with the boundary of the game's state space in finite time, the folk theorem of evolutionary game theory does not hold in a Euclidean context, even when the game is a potential one \cite{San10}.}

Despite these important rationality properties, the replicator dynamics fail to eliminate weakly dominated strategies \cite{Sam93};
furthermore, as is the case with all \emph{first-order} game dynamics \cite{HMC03}, they do not converge to Nash equilibrium in all games.
Thus, motivated by the success of second-order, ``heavy ball with friction'' methods in optimization \cite{Alv00,AABR02,Ant94,AGR00,HJ98,Pol87},
our first goal in this paper is to examine whether it is possible to obtain better convergence properties and/or escape the first-order impossibility results of \cite{HMC03} in a second-order setting.

To that end,
if we replace $\dot y$ by $\ddot y$ in \eqref{eq:EW}, we obtain the dynamics:
\begin{equation}
\label{eq:EW-2}
\tag{EW$_{2}$}
\begin{aligned}
\ddot y_{k\alpha}
	&= \payv_{k\alpha},
	\\
x_{k\alpha}
	&= \frac{\exp(y_{k\alpha})}{\sum_{\beta} \exp(y_{k\beta})}.
\end{aligned}
\end{equation}
These second-order exponential learning dynamics were studied in the very recent paper \cite{LM13} where it was shown that \eqref{eq:EW-2} is equivalent to the \emph{second-order replicator equation}:
\begin{equation}
\label{eq:RD-2}
\tag{RD$_{2}$}
\begin{aligned}
\ddot x_{k\alpha}
	&= x_{k\alpha}
	\left[
	\payv_{k\alpha}(x) - \insum_{\beta\in\act_{k}} x_{k\beta} \payv_{k\beta}(x)
	\right]
	\\
	&+ x_{k\alpha} \left[
	\dot x_{k\alpha}^{2}\big/ x_{k\alpha}^{2} - \insum_{\beta\in\act_{k}} \dot x_{k\beta}^{2}\big/x_{k\beta}
	\right].
\end{aligned}
\end{equation}
Importantly, under \eqref{eq:EW-2}/\eqref{eq:RD-2}, even \emph{weakly} dominated strategies become extinct;
such strategies may survive in perpetuity under the first-order dynamics \eqref{eq:RD}, so this represents a marked advantage for using second-order methods in games.

That being said, the second-order system \eqref{eq:RD-2} has no obvious ties to the gradient ascent properties of its first-order counterpart, so it is not clear whether its trajectories converge to Nash equilibrium in potential games.
On that account, a natural way to regain this connection would be to see whether \eqref{eq:RD-2} can be linked to the ``heavy ball with friction'' system:
\begin{equation}
\label{eq:HBF-S}
\frac{D^{2} x_{k}}{Dt^{2}}
	= \grad_{k}^{S} \pay_{k}(x)
	- \friction \dot x_{k},
\end{equation}
where
$\frac{D^{2}x_{k}}{Dt^{2}}$ denotes the covariant acceleration of $x_{k}$ under the Shahshahani metric
and
$\friction \geq 0$ is a friction coefficient, included in \eqref{eq:HBF-S} to slow down trajectories and enable convergence.
In this way, if the game admits a potential function $\pot$, the total energy $E(x,\dot x) = \frac{1}{2}\norm{\dot x}^{2} - \pot(x)$ will be Lyapunov under \eqref{eq:HBF-S} (by construction), so \eqref{eq:HBF-S} is intuitively expected to converge to the set of Nash equilibria of the game that are local maximizers of $\pot$.


Writing everything out in components (see Section \ref{sec:dynamics} for detailed definitions and derivations), we obtain the \emph{inertial replicator dynamics}:%
\footnote{We are very grateful to Jérôme Bolte for suggesting the term ``inertial''.} 
\begin{equation}
\label{eq:IRD}
\tag{IRD}
\begin{aligned}
\ddot x_{k\alpha}
	&= x_{k\alpha}
	\left[
	\payv_{k\alpha}(x) - \insum_{\beta\in\act_{k}} x_{k\beta} \payv_{k\beta}(x)
	\right]
	\\
	&+ \frac{1}{2} x_{k\alpha}
	\left[
	\dot x_{k\alpha}^{2}\big/ x_{k\alpha}^{2} - \insum_{\beta\in\act_{k}} \dot x_{k\beta}^{2}\big/x_{k\beta}
	\right]
	- \friction \dot x_{k\alpha},
\end{aligned}
\end{equation}
with the ``inertial'' velocity-dependent term of \eqref{eq:IRD} stemming from covariant differentiation under the Shahshahani metric.
Rather surprisingly (and in stark contrast to the first-order case),
we see that \eqref{eq:EW-2} and \eqref{eq:HBF-S} lead to dynamics that are similar but \emph{not} identical:
in the baseline, frictionless case ($\friction=0$), \eqref{eq:RD-2} and \eqref{eq:IRD} differ by a factor of $1/2$ in their velocity-dependent terms.
Further, in an even more surprising twist, this seemingly innocuous factor actually leads to drastic differences:
solutions to \eqref{eq:IRD} typically fail to exist for all time, so the rationality properties of the first- and second-order replicator dynamics do not (in fact, \emph{cannot}) extend to \eqref{eq:IRD}.

The reason that \eqref{eq:IRD} fails to be well-posed is deeply geometric and has to do with the fact that the Shahshahani simplex is isometric to an orthant of the Euclidean sphere (a bounded set that cannot restrain second-order ``heavy ball'' trajectories).
On that account, the second main goal of our paper is to examine whether the ``heavy ball with friction'' optimization principle that underlies \eqref{eq:HBF-S} can lead to a well-posed system with good rationality properties \emph{under a different choice of geometry}.

To that end, we focus on the class of \ac{HR} metrics \cite{Dui01,Shi77} that have been studied extensively in the context of convex programming \cite{ABB04,BT03};
in fact, the proposed class of dynamics provides a second-order, inertial extension of the gradient-like dynamics of \cite{BT03} to a game-theoretic setting with several agents, each seeking to maximize their individual payoff function.
The reason for focusing on the class of \ac{HR} metrics is that they are generated by taking the Hessian of a steep, strongly convex function over the problem's state space (a simplex-like object in our case), so, thanks to the geometry's ``steepness'' at the boundary of the feasible region, the induced first-order gradient flows are well-posed.
Of course, as the Shahshahani case shows,%
\footnote{In a certain sense, the Shahshahani metric (and the induced replicator dynamics) is the archetypal \acl{HR} metric, obtained by taking the Hessian of the Gibbs negative entropy.}
this ``steepness'' is not enough to guarantee well-posedness in a second-order setting;
however, if the geometry is ``steep enough'' (in a certain, precise sense), the resulting dynamics are well-posed and exhibit a fair set of long-term rationality properties (including convergence to equilibrium in the class of potential games).


The breakdown of our analysis is as follows:
in Section \ref{sec:dynamics}, we present an explicit derivation of the class of inertial game dynamics under study and we discuss their ``energy minimization'' properties in the class of potential games.
Our asymptotic analysis begins in Section \ref{sec:asymptotics} where we discuss the well-posedness problems that arise in the case of the replicator dynamics and we derive a geometric characterization of the \ac{HR} structures that lead to a well-posed flow:
as it turns out, global solutions exist if and only if the interior of the game's strategy space can be mapped isometrically to a closed (but not compact) hypersurface of some ambient Euclidean space.

Our rationality and convergence results are presented in Section \ref{sec:results}.
First, from an optimization viewpoint, we show that isolated maximizers of smooth functions defined on simplex-like objects are asymptotically stable;
as a result, Nash equilibria that are potential maximizers are asymptotically stable in potential games.
More generally, we establish the following ``folk theorem'' for general (multi-population) games:
\begin{inparaenum}[\itshape a\upshape)]
\item
Nash equilibria are stationary;
\item
if an interior orbit converges, its limit is a restricted equilibrium;
and
\item
strict equilibria attract all nearby trajectories.
\end{inparaenum}
Finally, in the framework of symmetric, single-population games, we show that \acp{ESS} are asymptotically stable in doubly symmetric games, providing in this way an extension of the corresponding result for the standard (single-population) replicator dynamics \cite{HS98};
by contrast, this result does not hold under the second-order replicator dynamics \eqref{eq:RD-2}.

For completeness, some elements of Riemannian geometry are discussed in Appendix \ref{app:geometry} (mostly to fix terminology and notation);
finally, to streamline the flow of ideas in the paper, some proofs and calculations have been delegated to Appendix \ref{app:calculations}.

\subsection{Notational conventions}
\label{sec:notation}

If $W$ is a vector space, we will write $W^{\ast}$ for its dual and $\braket{\omega}{w}$ for the pairing between the primal vector $w\in W$ and the dual vector $\omega\in W^{\ast}$.
By contrast, an inner product on $W$ will be denoted by $\product{\argdot}{\argdot}$, writing e.g. $\product{w}{w'}$ for the product between the (primal) vectors $w,w'\in W$.

The real space spanned by the finite set $\set = \{s_{\alpha}\}_{\alpha=0}^{n}$ will be denoted by $\R^{\set}$ and we will write $\{\bvec_{s}\}_{s\in\set}$ for its canonical basis.
In a slight abuse of notation, we will also use $\alpha$ to refer interchangeably to either $s_{\alpha}$ or $\bvec_{\alpha}$
and we will write $\delta_{\alpha\beta}$ for the Kronecker delta symbols on $\set$.
The set $\simplex(\set)$ of probability measures on $\set$ will be identified with the $n$-dimensional simplex $\simplex = \{x\in \R^{\set}: \sum_{\alpha} x_{\alpha} =1 \text{ and }x_{\alpha}\geq 0\}$ of $\R^{\set}$ and the relative interior of $\simplex$ will be denoted by $\intsimplex$.
Finally, if $\{\set_{k}\}_{k\in\play}$ is a finite family of finite sets, we will use the shorthand $(\alpha_{k};\alpha_{-k})$ for the tuple $(\dotsc,\alpha_{k-1},\alpha_{k},\alpha_{k+1},\dotsc)$;
also, when there is no danger of confusion, we will write $\sum_{\alpha}^{k}$ instead of $\sum_{\alpha\in\set_{k}}$.

\subsection{Definitions from game theory}
\label{sec:games}

A \emph{finite game in normal form} is a tuple $\game \defeq \game(\play,\act,\pay)$ consisting of
\begin{inparaenum}[\itshape a\upshape)]
\item
a finite set of \emph{players} $\play = \{1,\dotsc,N\}$;
\item
a finite set $\act_{k}$ of \emph{actions} (or \emph{pure strategies}) per player $k\in\play$;
and
\item
the players' \emph{payoff functions} $\pay_{k}\from \act\to \R$,
where $\act\equiv\prod_{k}\act_{k}$ denotes the set of all joint action profiles $(\alpha_{1},\dotsc,\alpha_{N})$.
\end{inparaenum}
The set of \emph{mixed strategies} of player $k$ will be denoted by $\strat_{k}\equiv\simplex(\act_{k})$ and we will write $\strat\equiv\prod_{k}\strat_{k}$ for the game's \emph{state space} \textendash\ i.e. the space of \emph{mixed strategy profiles} $x = (x_{1},\dotsc,x_{N})$.
Unless mentioned otherwise, we will write $V_{k} \equiv \R^{\act_{k}}$ and $V \equiv \prod_{k} V_{k} \cong \R^{\coprod_{k} \act_{k}}$ for the ambient spaces of $\strat_{k}$ and $\strat$ respectively.

The \emph{expected payoff} of player $k$ in the strategy profile $x = (x_{1},\dotsc,x_{N})\in \strat$ is
\begin{equation}
\label{eq:pay}
\pay_{k}(x)
	= \insum_{\alpha_{1}}^{1}\dotsi \insum_{\alpha_{N}}^{N}
	\pay_{k}(\alpha_{1},\dotsc,\alpha_{N}) \; x_{1,\alpha_{1}} \dotsm\, x_{N,\alpha_{N}},
\end{equation}
where $\pay_{k}(\alpha_{1},\dotsc,\alpha_{N})$ denotes the payoff of player $k$ in the pure profile $(\alpha_{1},\dotsc,\alpha_{N})\in\act$.
Accordingly, the payoff corresponding to $\alpha\in\act_{k}$ in the mixed profile $x\in\strat$ is
\begin{flalign}
\label{eq:payv}
\payv_{k\alpha}(x)
	&= \insum_{\alpha_{1}}^{1} \dotsi \insum_{\alpha_{N}}^{N}
	\pay_{k}(\alpha_{1},\dotsc,\alpha_{N})
	\; x_{1,\alpha_{1}} \dotsm\,\delta_{\alpha_{k},\alpha} \dotsm\, x_{N,\alpha_{N}},
\end{flalign}
and we have
\begin{equation}
\label{eq:pay-pairing}
\pay_{k}(x)
	= \insum_{\alpha}^{k} x_{k\alpha} \payv_{k\alpha}(x)
	= \braket{\payv_{k}(x)}{x_{k}}
\end{equation}
where $\payv_{k}(x) = (\payv_{k\alpha}(x))_{\alpha\in\act_{k}}$ denotes the \emph{payoff vector} of player $k$ at $x\in\strat$.

In the above, $\payv_{k}$ is treated as a dual vector in $V_{k}^{\ast}$ that is paired to the mixed strategy $x_{k}\in\strat_{k}$;
on that account, mixed strategies will be regarded throughout this paper as \emph{primal} variables and payoff vectors as \emph{duals}.
Moreover, note that $\payv_{k\alpha}(x)$ does not depend on $x_{k\alpha}$ so we have $\payv_{k\alpha} = \frac{\pd\pay_{k}}{\pd x_{k\alpha}}$;
in view of this, we will often refer to $\payv_{k\alpha}$ as the \emph{marginal utility} of action $\alpha\in\act_{k}$ and we will identify $\payv_{k}(x)\in V^{\ast}$ with the (unilateral) differential of $\pay_{k}(x)$ with respect to $x_{k}$.

Finally, following \cite{MS96,San01}, we will say that $\game$ is a \emph{potential game} when it admits a potential function $\pot\from\strat\to\R$ such that:
\begin{equation}
\label{eq:potential}
\payv_{k\alpha}(x)
	= \frac{\pd\pot}{\pd x_{k\alpha}}
	\quad
	\text{for all $x\in\strat$ and for all $\alpha\in\act_{k}$, $k\in\play$,}
\end{equation}
or, equivalently:
\begin{equation}
\label{eq:potential-diff}
\pay_{k}(x_{k};x_{-k}) - \pay_{k}(x_{k}';x_{-k})
	= \pot(x_{k};x_{-k}) - \pot(x_{k}';x_{-k}),
\end{equation}
for all $x_{k}\in\strat_{k}$ and for all $x_{-k}\in\strat_{-k}\equiv\prod_{\ell\neq k} \strat_{\ell}$, $k\in\play$.

%% file: Dynamics.tex

In this section, we introduce the class of inertial game dynamics that comprise the main focus of our paper.
For notational simplicity, most of our derivations are presented in the case of a single player with a finite action set $\act = \{0,\dotsc,n\}$;
the extension to the general, multi-player case is straightforward and simply involves reinstating the player index $k$ where necessary.

As we explained in the introduction, the dynamics under study in this unilateral framework boil down to the ``heavy ball with friction'' system:
\begin{equation}
\label{eq:HBF}
\tag{HBF}
\frac{D^{2}x}{Dt^{2}}
	= \grad\pay(x)
	- \friction \dot x,
\end{equation}
where gradients and covariant derivatives are taken with respect to a Riemannian metric $g$ on the game's state space $\strat \equiv \simplex(\act)$ \textendash\ for a brief discussion of the necessary concepts from Riemannian geometry, the reader is referred to Appendix \ref{app:geometry}.
Of course, in the ordinary Euclidean case (where covariant and ordinary derivatives coincide), there is no barrier term in \eqref{eq:HBF} that can constrain the dynamics' solution trajectories to remain in $\strat$ for all time;
as such, we begin by presenting a class of Riemannian metrics with a more appropriate boundary behavior.

\subsection{\acl{HR} metrics}
\label{sec:HR}

Following Bolte and Teboulle \cite{BT03} and Alvarez et al. \cite{ABB04}, we begin by endowing the positive orthant $\orthalt \equiv \orthant^{\act} \equiv \{x\in \R^{\act}: x_{\alpha}>0\}$ of the ambient space $V=\R^{\act}$ of $\strat$ with a Riemannian metric $g(x)$ that blows up at the boundary hyperplanes $x_{\alpha}=0$ \textendash\ raising in this way an inherent geometric barrier on the boundary $\bd(\strat)$ of $\strat$.

A standard device to achieve this blow-up is to define $g(x)$ as the Hessian of a strongly convex function $h\from\orthalt\to\R$ that becomes infinitely steep at the boundary of $\orthalt$ \cite{ABB04,BT03,MS14,SS11}.
To that end, let $\theta\from[0,+\infty) \to \R\cup\{+\infty\}$ be a $C^{\infty}$-smooth function satisfying the Legendre-type properties \cite{ABB04,BT03,Roc70}:%
\footnote{Legendre-type functions are usually defined without the regularity requirement $\theta'''<0$.
This assumption can be relaxed without significantly affecting our results but we will keep it for simplicity.}
\begin{equation}
\label{eq:kernel}
\tag{L}
\begin{aligned}
\text{1.}\quad
	& \text{$\theta(x) < \infty $ for all $x>0$.}
\hspace{200pt}\\[-1pt]
\text{2.}\quad
	& \txs\lim_{x\to0^{+}}\theta'(x) = -\infty.\\
\text{3.}\quad
	& \text{$\theta''(x) >0$ and $\theta'''(x)<0$ for all $x>0$.}
\end{aligned}
\end{equation}
We then define the associated \emph{penalty function}
\begin{equation}
\label{eq:penalty}
h(x)
	= \sum_{\alpha=0}^{n} \theta(x_{\alpha}),
\end{equation}
and we define a metric $g$ on $\orthalt$ by taking the Hessian of $h$, viz.:
\begin{equation}
\label{eq:HR}
g_{\alpha\beta}
	= \frac{\pd^{2}h}{\pd x_{\alpha} \pd x_{\beta}}
	= \theta_{\alpha}'' \delta_{\alpha\beta},
\end{equation}
where the shorthand $\theta_{\alpha}''$, $\alpha=0,\dotsc,n$, stands for $\theta''(x_{\alpha})$.
In other words, the \emph{\acl{HR} metric induced by $\theta$} is the field of positive-definite matrices
\begin{equation}
g(x)
	= \diag(\theta''(x_{0}),\dotsc,\theta''(x_{n})),
\quad
x\in\orthalt.
\end{equation}
With $h$ strictly convex (recall that $\theta''>0$), it follows that $g$ is indeed a Riemannian metric tensor on $\orthalt$;
following \cite{ABB04}, we will refer to $\theta$ as the \emph{kernel} of $g$.

\begin{remark}
The penalty function $h$ of \eqref{eq:penalty} is closely related to the class of \emph{control cost} functions used to define quantal responses in the theory of discrete choice \cite{MP95,vD87} and the class of \emph{regularizer functions }used in mirror descent methods for optimization and online learning \cite{MS14,NY83,Nes09,SS11};
for a detailed discussion, we refer the reader to \cite{ABB04,BT03,CGM15}.
In fact, more general \acl{HR} structures can be obtained by considering $C^{2}$-smooth strongly convex functions $h\from\orthalt\to\R$ that do not necessarily admit a decomposition of the form \eqref{eq:penalty}.
Most of our results can be extended to this non-separable setting but the calculations involved are significantly more tedious, so we will focus on the simpler, decomposable framework of \eqref{eq:penalty}.%
\footnote{In particular, the results that do not hold verbatim are those that call explicitly on $\theta$ \textendash\ most notably, Corollary \ref{cor:wp}.}
\end{remark}

\begin{example}[The Shahshahani metric]
The most widely studied example of a non-Euclidean \ac{HR} structure on the simplex is generated by the entropic kernel $\theta^{S}(x) = x \log x$.
By differentiation, we then obtain the \emph{Shahshahani metric} \cite{Aki79,ABB04,Sha79}:
\begin{equation}
\label{eq:Shah-matrix}
g^{S}(x)
	= \diag(1/x_{0},\dotsc,1/x_{n}),
	\quad
	x\in\orthalt,
\end{equation}
or, in coordinates:
\begin{equation}
\label{eq:Shah}
g_{\alpha\beta}^{S}(x)
	= \delta_{\alpha\beta}/x_{\beta}.
\end{equation}
\end{example}

\begin{example}[The log-barrier]
Another important example with close ties to proximal barrier methods in optimization (see e.g. \cite{ABB04,BT03} and references therein) is given by the logarithmic barrier kernel $\theta^{L}(x) = -\log x$ \cite{ABB04,BL89,Fia90,McC89}.
The associated penalty function is $h(x) = -\insum_{\alpha} \log x_{\alpha}$ and its Hessian generates the metric
\begin{equation}
\label{eq:log}
g_{\alpha\beta}^{L}(x)
	= \delta_{\alpha\beta}/x_{\beta}^{2},
\end{equation}
or, in matrix form:
\begin{equation}
\label{eq:log-matrix}
g^{L}(x)
	= \diag(1/x_{0}^{2},\dotsc, 1/x_{n}^{2}),
	\quad
	x\in\orthalt.
\end{equation}
An important qualitative difference between the kernels $\theta^{S}$ and $\theta^{L}$ is that the former remains bounded as $x\to0^{+}$ whereas the latter blows up;
this difference will play a key role with regard to the existence of global solutions.
\end{example}

\subsection{Derivation of the dynamics and examples}
\label{sec:calculations}

Having endowed $\orthalt$ with a \acl{HR} structure $g$ with kernel $\theta$, we continue with the calculation of the gradient and acceleration terms of \eqref{eq:HBF}.
To that end, it will be convenient to introduce the coordinate transformation
\begin{equation}
\label{eq:pi0}
\pi_{0} \from (x_{0},x_{1},\dotsc,x_{n})
	\mapsto (x_{1},\dotsc,x_{n}),
\end{equation}
which maps the affine hull of $\strat$ isomorphically to $V_{0} \equiv \R^{n}$ by eliminating $x_{0}$.
The (right) inverse of this transformation is given by the injection
\begin{equation}
\label{eq:iota0}
\txs
\iota_{0} \from (x_{1},\dotsc,x_{n})
	\mapsto (1-\insum_{\alpha=1}^{n} x_{\alpha}, x_{1},\dotsc,x_{n}),
\end{equation}
so \eqref{eq:pi0} provides a global coordinate chart for $\strat$ that will allow us to carry out the necessary geometric calculations.

As a first step, let $\{\bvec_{\alpha}\}_{\alpha=0}^{n}$ and $\{\tilde\bvec_{\mu}\}_{\mu=1}^{n}$ denote the canonical bases of $V$ and $V_{0}$ respectively.
Then, under $\iota_{0}$, $\tilde\bvec_{\mu}$ is pushed forward to $(\iota_{0})_{\ast} \tilde\bvec_{\mu} = \bvec_{\mu} - \bvec_{0}$,%
\footnote{Simply note that the image of the coordinate curve $\gamma_{\mu}(t) = t \tilde\bvec_{\mu}$ under $\iota_{0}$ is $-t\bvec_{0} + t\bvec_{\mu}$.}
so the component-wise expression of $g$ in the coordinates \eqref{eq:pi0} is:
\begin{equation}
\label{eq:metric-reduced}
\tilde g_{\mu\nu}
	= \product{\bvec_{\mu} - \bvec_{0}}{\bvec_{\nu} - \bvec_{0}}
	= g_{\mu\nu} + g_{00}
	= \theta_{\mu}'' \delta_{\mu\nu} +\theta_{0}''.
\end{equation}
With this coordinate expression at hand, let $f\from\intstrat\to\R$ be a (smooth) function on $\intstrat$ and write $\tilde f = f\circ\iota_{0}$, $(x_{1},\dotsc,x_{n}) \mapsto f(1 - \insum_{\alpha=1}^{n},x_{1},\dotsc,x_{n})$ for its coordinate expression under \eqref{eq:iota0}.
Referring to Appendix \ref{app:geometry} for the required background definitions,%
\footnote{We only mention here that $\grad f$ is characterized by the chain rule property $\frac{d}{dt} f(x(t)) = \product{\dot x(t)}{\grad f(x(t))}$ for every smooth curve $x(t)$ on $\intstrat$.}
the gradient of $f$ with respect to $g$ may be expressed as:
\begin{equation}
\label{eq:gradient-coords1}
\grad f
	= g^{-1}\cdot\nabla f
	= \sum_{\mu,\nu=1}^{n} \tilde g^{\mu\nu} \frac{\pd \tilde f}{\pd x_{\nu}} \tilde\bvec_{\mu}
\end{equation}
where $\tilde g^{\mu\nu}$ is the inverse matrix of $\tilde g_{\mu\nu}$.
By the inversion formula of Lemma \ref{lem:inversion}, we then obtain
\begin{equation}
\label{eq:metric-inverse}
\tilde g^{\mu\nu}
	= \frac{\delta_{\mu\nu}}{\theta_{\mu}''}
	- \frac{\Theta''}{\theta_{\mu}'' \theta_{\nu}''},
\end{equation}
where $\Theta'' = \big( \insum_{\beta} 1/\theta_{\beta}'' \big)^{-1}$ denotes the ``harmonic sum'' of the metric weights $\theta_{\beta}''$.%
\footnote{Note that $\Theta''$ is not a second derivative; we are only using this notation for visual consistency.}
Thus, by carrying out the summation in \eqref{eq:gradient-coords1}, we get the coordinate expression:
\begin{equation}
\label{eq:gradient-coords2}
\grad f
	= \sum_{\mu=1}^{n} \frac{1}{\theta_{\mu}''}
	\bigg[
	\frac{\pd\tilde f}{\pd x_{\mu}}
	- \Theta'' \sum_{\nu=1}^{n} \frac{1}{\theta_{\nu}''} \frac{\pd\tilde f}{\pd x_{\nu}}
	\bigg]
	\tilde\bvec_{\mu}.
\end{equation}
Accordingly, if the domain of $f\from\intstrat\to\R$ extends to an open neighborhood of $\intstrat$ (so $\pd_{\mu}\tilde f = \pd_{\mu} f - \pd_{0}f$ for all $x\in\intstrat$), some algebra readily gives:
\begin{equation}
\label{eq:gradient-coords}
\grad f
	= \sum_{\alpha=0}^{n} \frac{1}{\theta_{\alpha}''}
	\bigg[
	\frac{\pd f}{\pd x_{\alpha}}
	- \Theta'' \sum_{\beta=0}^{n} \frac{1}{\theta_{\beta}''} \frac{\pd f}{\pd x_{\beta}}
	\bigg]
	\bvec_{\alpha}.
\end{equation}

With regard to the inertial acceleration term of \eqref{eq:HBF}, taking the covariant derivative of $\dot x$ in the coordinate frame \eqref{eq:pi0} yields:
\begin{equation}
\label{eq:acceleration-coords1}
\frac{D^{2} x_{\mu}}{Dt^{2}}
	= \ddot x_{\mu}
	+ \sum_{\nu,\rho=1}^{n} \tilde\Gamma^{\mu}_{\nu\rho} \dot x_{\nu} \dot x_{\rho},
	\quad
	\mu =1,\dotsc,n,
\end{equation}
where the so-called \emph{Christoffel symbols} $\tilde\Gamma_{\nu\rho}^{\mu}$ of $g$ are given by:%
\footnote{For a more detailed discussion the reader is again referred to Appendix \ref{app:geometry}.
We only mention here that the covariant derivative in \eqref{eq:acceleration-coords1} is defined so that the system's energy $E(x,\dot x) = \frac{1}{2}\norm{\dot x}^{2} - u(x)$ is a constant of motion under \eqref{eq:HBF} when $\friction=0$ (and Lyapunov when $\friction>0$).}
\begin{equation}
\label{eq:Christoffel-2}
\tilde\Gamma_{\nu\rho}^{\mu}
	= \frac{1}{2} \sum_{\kappa=1}^{n} \tilde g^{\mu\kappa}
	\left(
	\frac{\pd \tilde g_{\kappa\nu}}{\pd x_{\rho}}
	+ \frac{\pd \tilde g_{\rho\kappa}}{\pd x_{\nu}}
	- \frac{\pd \tilde g_{\nu\rho}}{\pd x_{\kappa}}	
	\right).
\end{equation}
After a somewhat cumbersome calculation (cf. Appendix \ref{app:calculations}), we then get:
\begin{equation}
\label{eq:acceleration-coords2}
\frac{D^{2}x_{\mu}}{Dt^{2}}
	= \ddot x_{\mu}
	+ \frac{1}{2} \frac{\theta_{\mu}'''}{\theta_{\mu}''} \dot x_{\mu}^{2}
	- \frac{1}{2} \frac{\Theta''}{\theta_{\mu}''}
	\left[
	\sum_{\nu=1}^{n} \frac{\theta_{\nu}'''}{\theta_{\nu}''} \dot x_{\nu}^{2}
	+ \frac{\theta_{0}'''}{\theta_{0}''} \left(\sum_{\nu=1}^{n} \dot x_{\nu}\right)^{2}
	\right],
\end{equation}
so, with $\dot x_{0} = -\insum_{\nu=1}^{n} \dot x_{\nu}$, \eqref{eq:acceleration-coords1} becomes:
\begin{equation}
\label{eq:acceleration-coords}
\frac{D^{2}x_{\alpha}}{Dt^{2}}
	= \ddot x_{\alpha}
	+ \frac{1}{2} \frac{1}{\theta_{\alpha}''}
	\bigg[
	\theta_{\alpha}''' \dot x_{\alpha}^{2}
	- \sum_{\beta=0}^{n} \left(\Theta''\big/\theta_{\beta}''\right) \theta_{\beta}''' \dot x_{\beta}^{2}
	\bigg].
\end{equation}

In view of the above, putting together \eqref{eq:gradient-coords}, \eqref{eq:acceleration-coords} and \eqref{eq:HBF}, we obtain the \emph{inertial game dynamics}:
\begin{equation}
\label{eq:ID}
\tag{ID}
\begin{aligned}
\ddot x_{k\alpha}
	&= \frac{1}{\theta_{k\alpha}''}
	\bigg[
	\payv_{k\alpha} - \insum_{\beta}^{k} \left(\Theta_{k}''\big/\theta_{k\beta}''\right) \payv_{k\beta}
	\bigg]
	\notag\\
	&- \frac{1}{2} \frac{1}{\theta_{k\alpha}''}
	\bigg[
	\theta_{k\alpha}''' \dot x_{k\alpha}^{2} - \insum_{\beta}^{k} \left(\Theta_{k}''\big/\theta_{k\beta}''\right) \theta_{k\beta}''' \dot x_{k\beta}^{2}
	\bigg]
	-\friction \dot x_{k\alpha},
\end{aligned}
\end{equation}
where, in obvious notation, we have reinstated the player index $k$ and we have used the fact that $\payv_{k\alpha} = \frac{\pd\pay_{k}}{\pd x_{k\alpha}}$.
Since these dynamics comprise the main focus of our paper, we immediately proceed to two representative examples:

\smallskip

\begin{example}[The inertial replicator dynamics]
The Shahshahani kernel $\theta(x) = x\log x$ has $\theta''(x) = 1/x$ and $\theta'''(x) = -1/x^{2}$, so \eqref{eq:ID} leads to the \emph{inertial replicator dynamics}:
\begin{equation}
\tag{IRD}
\ddot x_{k\alpha}
	= x_{k\alpha}
	\left[
	\payv_{k\alpha} - \insum_{\beta}^{k} x_{k\beta} \payv_{k\beta}
	\right]
	+ \frac{1}{2} x_{k\alpha}
	\left[
	\dot x_{k\alpha}^{2}\big/ x_{k\alpha}^{2} - \insum_{\beta}^{k} \dot x_{k\beta}^{2}\big/x_{k\beta}
	\right]
	-\friction \dot x_{k\alpha}.
\end{equation}
As we mentioned in the introduction, the only notable difference between \eqref{eq:IRD} and the second-order replicator dynamics of exponential learning \eqref{eq:RD-2} is the factor $1/2$ in the RHS of \eqref{eq:IRD} (the friction term $\friction\dot x$ is not important for this comparison).
Despite the innocuous character of this scaling-like factor,%
\footnote{It is tempting to interpret the factor $1/2$ in \eqref{eq:IRD} as a change of time with respect to \eqref{eq:RD-2}, but the presence of $\dot x^{2}$ precludes as much.}
we shall see in the following section that \eqref{eq:IRD} and \eqref{eq:RD-2} behave in drastically different ways.
\end{example}

\begin{example}[The inertial log-barrier dynamics]
The log-barrier kernel $\theta(x) = -\log x$ has $\theta''(x) = 1/x^{2}$ and $\theta'''(x) = -2/x^{3}$, so we obtain the \emph{inertial log-barrier dynamics}:
\begin{equation}
\label{eq:ILD}
\tag{ILD}
\ddot x_{k\alpha}
	= x_{k\alpha}^{2}
	\left[
	\payv_{k\alpha} - r_{k}^{-2} \insum_{\beta}^{k} x_{k\beta}^{2} \payv_{k\beta}
	\right]
	+ x_{k\alpha}^{2}
	\left[
	\dot x_{k\alpha}^{2}\big/ x_{k\alpha}^{3} - r_{k}^{-2}\insum_{\beta}^{k} \dot x_{k\beta}^{2}\big/x_{k\beta}
	\right]
	- \friction \dot x_{k\alpha},
\end{equation}
where $r_{k}^{2} = \insum_{\beta}^{k} x_{k\beta}^{2}$.
The first order analogue of these dynamics \textendash\ namely, the system $\dot x_{k\alpha} = x_{k\alpha}^{2} \big(\payv_{k\alpha} - r_{k}^{-2} \insum_{k\beta} x_{k\beta}^{2} \payv_{k\beta} \big)$ \textendash\ has been studied extensively in the context of linear programming and convex optimization \cite{ABB04,BL89,BT03,Fia90,McC89}, while its game-theoretic properties are discussed in \cite{MS14}.
\end{example}

%% file: Asymptotics.tex

In this section, we examine the energy dissipation and well-posedness properties of \eqref{eq:ID}.
For convenience, we will work with the single-agent version of the dynamics \eqref{eq:ID} with $\payv = \nabla\pot$ for some Lipschitz continuous and sufficiently smooth function $\pot$ on $\strat$.%
\footnote{Here and in what follows, it will be convenient to assume that $\pot$ is defined on an open neighborhood of $\strat$.
This assumption facilitates the use of standard coordinates for calculations, but none of our results depend on this device.}%

\subsection{Friction and Dissipation of Energy}
\label{sec:friction}

We begin by showing that the system's total energy
\begin{equation}
\label{eq:energy}
E(x,\dot x)
	= \frac{1}{2} \norm{\dot x}^{2} - \pot(x)
\end{equation}
is dissipated along the inertial dynamics \eqref{eq:ID} for $\friction>0$ (or is a constant of motion in the frictionless case $\friction=0$).

\begin{proposition}
\label{prop:dissipation}
The total energy $E(x,\dot x)$ is nonincreasing along any interior solution orbit of \eqref{eq:ID};
specifically:
\begin{equation}
\label{eq:dissipation}
\dot E
	= - 2\friction \kin
	= -\friction \norm{\dot x}^{2},
\end{equation}
where $\kin = \frac{1}{2}\norm{\dot x}^{2}$ is the system's kinetic energy.
\end{proposition}

\begin{proof}
By differentiating \eqref{eq:energy} along $\dot x(t)$, we readily obtain:
\begin{flalign}
\dot E
	& = \del_{\dot x} E
	= \tfrac{1}{2} \del_{\dot x} \product{\dot x}{\dot x}
	- \del_{\dot x} \pot
	= \product{\del_{\dot x} \dot x}{\dot x}
	- \braket{d\pot}{\dot x}
	\notag\\
	& = \product{\frac{D^{2}x}{Dt^{2}}}{\dot x}
	- \product{\grad\pot}{\dot x}
	= \product{\grad\pot - \friction \dot x}{\dot x}
	- \product{\grad\pot}{\dot x}
	\notag\\
	& = -\friction \product{\dot x}{\dot x}
	= - 2 \friction \kin,
\end{flalign}
where we used the metric compatibility \eqref{eq:compatibility} of $\del$ in the first line, and the definition of the dynamics \eqref{eq:ID} in the second.
\hfill
\end{proof}

Proposition \ref{prop:dissipation} shows that, for $\friction>0$, the system's total energy $E = \kin - \pot$ is a Lyapunov function for \eqref{eq:ID};
by contrast, in first-order \ac{HR} gradient flows \cite{ABB04,BT03}, it is the maximization objective $\pot$ that acts as a Lyapunov function.
As such, in the second-order context of \eqref{eq:ID}, it will be important to show that the system's kinetic energy eventually vanishes \textendash\ so that $\pot$ becomes an ``asymptotic'' Lyapunov function.
To that end, we have:

\begin{proposition}
\label{prop:stop}
Let $x(t)$ be a solution trajectory of \eqref{eq:ID} that is defined for all $t\geq0$.
If $\friction>0$, then $\lim_{t\to\infty} \dot x(t) = 0$.
\end{proposition}

To prove Proposition \ref{prop:stop}, we will require the following intermediate result:

\begin{lemma}
\label{lem:kin-bound}
Let $x(t)$ be an interior solution of \eqref{eq:ID} that is defined for all $t\geq0$.
If $\friction>0$, the rate of change of the system's kinetic energy is bounded from above for all $t\geq0$.
\end{lemma}

\begin{proof}
By differentiating $\kin$ with respect to time, we readily obtain:
\begin{flalign}
\label{eq:kin-estimate}
\dot\kin
	&= \del_{\dot x} \kin
	= \product{\frac{D^{2}x}{Dt^{2}}}{\dot x}
	= \product{\grad\pot - \friction \dot x}{\dot x}
	\notag\\
	&= \braket{d\pot}{\dot x} - \friction \norm{\dot x}^{2}
	= \insum_{\beta} \frac{\pd\pot}{\pd x_{\beta}} \dot x_{\beta}
	- \friction \insum_{\beta} \theta_{\beta}'' \dot x_{\beta}^{2}
	\notag\\
	&\leq A \insum_{\beta} |\dot x_{\beta}| - \friction B \insum_{\beta} \dot x_{\beta}^{2},
\end{flalign}
where $A = \sup \abs{\pd_{\beta}\pot} < \infty$ and $B = \inf\{\theta''(x): x\in(0,1)\}$.
With $A$ finite and $B>0$ (on account of the Legendre properties of $\theta$), the maximum value of the above expression is $(n+1) A^{2}/(4\friction B)$, so $\dot\kin$ is bounded from above.
\hfill
\end{proof}

\begin{proofof}{Proof of Proposition \ref{prop:stop}}
Let $E(t) = \tfrac{1}{2} \norm{\dot x(t)}^{2} - \pot(x(t))$ be the system's energy at time $t$.
Proposition \ref{prop:dissipation} shows that $\dot E = -\friction \norm{\dot x}^{2} = -2\friction\kin\leq0$, so $E(t)$ decreases to some value $E^{\ast}\in\R$;
as a result, we also get $\int_{0}^{\infty} \kin(s) \dd s = (2\friction)^{-1} \left(E(0) - E^{\ast}\right)<\infty$.
This suggests that $\lim_{t\to\infty} \kin(t) = 0$, but since there exist positive integrable functions which do not converge to $0$ as $t\to\infty$, our assertion does not yet follow.

Assume therefore that $\limsup_{t\to\infty} \kin(t) = 3\eps >0$.
In that case, there exists by continuity an increasing sequence of times $t_{n}\to\infty$ such that $\kin(t_{n})>2\eps$ for all $n$.
Accordingly, let $s_{n} = \sup\{t: t\leq t_{n} \text{ and } \kin(t) <\eps\}$:
since $\kin$ is integrable and non-negative, we also have $s_{n}\to\infty$ (because $\liminf \kin(t) = 0$),
so, by descending to a subsequence of $t_{n}$, we may assume without loss of generality that $s_{n+1} > t_{n}$ for all $n$.
Hence, if we let $J_{n}=[s_{n},t_{n}]$, we have:
\begin{equation}
\int_{0}^{\infty} \kin(s) \dd s
	\geq \insum_{n=1}^{\infty} \int_{J_{n}} \kin
	\geq \eps \insum_{n=1}^{\infty} |J_{n}|,
\end{equation}
which shows that the Lebesgue measure $|J_{n}|$ of $J_{n}$ vanishes as $t\to\infty$.
Consequently, by the mean value theorem, it follows that there exists $\xi_{n}\in(s_{n},t_{n})$ such that
\begin{equation}
\dot \kin(\xi_{n})
	= \frac{\kin(t_{n}) - \kin(s_{n})}{|J_{n}|}
	> \frac{\eps}{|J_{n}|},
\end{equation}
and since $|J_{n}|\to 0$, we conclude that $\limsup \dot \kin(t) = \infty$.
This contradicts the conclusion of Lemma \ref{lem:kin-bound}, so we get $K(t)\to0$ and $\dot x(t)\to0$.
\hfill
\end{proofof}

\begin{remark}
The proof technique above easily extends to the Euclidean case, thus providing an alternative proof of the velocity integrability and convergence part of Theorem 2.1 in \cite{Alv00};
furthermore, if we consider a Hessian-driven damping term in \eqref{eq:ID} as in \cite{AABR02}, the estimate \eqref{eq:kin-estimate} remains essentially unchanged and our approach may also be used to prove the corresponding claim of Theorem 2.1 in \cite{AABR02}.
\end{remark}

\subsection{Well-posedness and Euclidean coordinates}
\label{sec:wp}

Clearly, Proposition \ref{prop:stop} applies if and only if the trajectory in question exists for all time, so it is crucial to determine whether the dynamics \eqref{eq:ID} are well-posed.
To that end, we will begin with the inertial replicator dynamics \eqref{eq:IRD} in the simple, baseline case $\strat=[0,1]$, $\pot=0$, which corresponds to a single player with two twin actions \textendash\ say $\act=\{0,1\}$ with $\payv_{0} = \payv_{1} = 0$.
Setting $x=x_{1} = 1-x_{0}$ in \eqref{eq:IRD}, we then get the second-order ODE:
\begin{equation}
\label{eq:IRD-simple}
\ddot x
	= \frac{1}{2} x \left(\dot x\big/ x^{2} - \dot x^{2} \big/ x - \dot x^{2}\big/(1-x)\right)
	= \frac{1}{2} \frac{1-2x}{x(1-x)} \dot x^{2}.
\end{equation}
To solve this equation, let $\xi = 2\sqrt{x}$ and $\vel = \dot \xi$;
after some algebra, we obtain the separable equation
\begin{equation}
\label{eq:IRD-transformed}
\frac{d\vel}{\vel}
	= -\frac{\xi}{4 - \xi^{2}} \dd \xi,
\end{equation}
which, after integrating, further reduces to:
\begin{equation}
\label{eq:xi-IRD}
\dot \xi
	= \vel
	= \frac{\vel_{0}}{4 - \xi_{0}^{2}} \sqrt{4 - \xi^{2}},
\end{equation}
with $\xi_{0} = \xi(0)$ and $\vel_{0} = \vel(0) = \dot\xi(0)$.
Some more algebra then yields the solution
\begin{equation}
\label{eq:oscillation}
\xi(t)
	= \xi_{0} \cos \frac{\vel_{0}t}{\sqrt{4 - \xi_{0}^{2}}}
	+ \sqrt{4 - \xi_{0}^{2}} \sin \frac{\vel_{0}t}{\sqrt{4-\xi_{0}^{2}}}.
\end{equation}
From the above, we see that $\xi(t)$ becomes negative in finite time for every interior initial position $\xi_{0}\in(0,2)$ and for all $\vel_{0}\in\R$.
However, since $\xi(t) = 2\sqrt{x(t)}$ by definition, this can only occur if $x(t)$ exits $(0,1)$ in finite time;
as a result, we conclude that the inertial replicator dynamics \eqref{eq:IRD} may fail to be well-posed, even in the simple case of the zero game.

On the other hand, a similar calculation for the inertial log-barrier dynamics \eqref{eq:ILD} yields the equation:
\begin{equation}
\label{eq:ILD-simple}
\ddot x
	= \dot x^{2} \frac{(1-2x)(1-x+x^{2})}{x(1-x)(1-2x+2x^{2})},
\end{equation}
which, after the change of variables $\xi = \log x<0$ (recall that $0<x<1$), becomes:
\begin{equation}
\ddot\xi
	= - \dot\xi^{2} \frac{e^{2\xi}}{(1-e^{\xi}) ( e^{2\xi} + (1-e^{\xi})^{2} \big)}.
\end{equation}
Setting $\vel = \dot\xi$ and separating as before, we then obtain the equation:
\begin{equation}
\label{eq:xi-ILD}
\dot\xi
	= C \frac{1-e^{\xi}}{\sqrt{1 - 2e^{\xi} + 2e^{2\xi}}},
\end{equation}
where $C\in\R$ is an integration constant.
Contrary to \eqref{eq:xi-ILD}, the RHS of \eqref{eq:xi-ILD} is Lipschitz and bounded for $\xi<0$ (and vanishes at $\xi=0$), so the solution $\xi(t)$ exists for all time;
as a result, we conclude that the simple system \eqref{eq:ILD-simple} is well-posed.

The fundamental difference between \eqref{eq:IRD-simple} and \eqref{eq:ILD-simple} is that the image of $(0,1)$ under the change of variables $x\mapsto 2\sqrt{x}$ is a bounded set whereas the image of the transformation $x\mapsto \log x$ is the (unbounded) half-line $\xi<0$:
consequently, the solutions of \eqref{eq:xi-IRD} escape from the image of $(0,1)$ in finite time, whereas the solutions of \eqref{eq:xi-ILD} remain contained therein for all $t\geq0$.
As we show below, this is a special case of a more general geometric principle which characterizes those \ac{HR} structures that lead to well-posed dynamics.

Our first step will be to construct a Euclidean equivalent of the dynamics \eqref{eq:IRD} by mapping $\intstrat$ isometrically in an ambient Euclidean space.
To that end, let $g$ be a Riemannian metric on the open orthant $\orthalt \equiv \orthant^{n+1}$ of $V \equiv \R^{n+1}$ and assume there exists a sufficiently smooth strictly convex function $\psi\from\orthalt\to\R$ such that:%
\footnote{We thank an anonymous reviewer for suggesting this synthetic approach.}
\begin{equation}
\label{eq:transform}
g
	= \hess(\psi)^{2}.
\end{equation}
Then, the derivative map $G\from\orthalt\to V$, $x\mapsto G(x) \equiv \nabla\psi(x)$, is
\begin{inparaenum}
[\itshape a\upshape)]
\item
injective (as the derivative of a strictly convex function);
and
\item
an immersion (since $\hess(\psi) \succ 0$).
\end{inparaenum}

Assume now that the target ambient space $V$ is endowed with the Euclidean metric $\delta(\bvec_{\alpha},\bvec_{\beta}) = \delta_{\alpha\beta}$;
we then claim that $G\from(\orthalt,g) \to (V,\delta)$ is an \emph{isometry}, i.e.
\begin{equation}
\label{eq:isometry}
	g(\bvec_{\alpha},\bvec_{\beta})
	= \delta(G_{\ast}\bvec_{\alpha},G_{\ast}\bvec_{\beta})
	\quad
	\text{for all $\alpha,\beta = 0,1,\dotsc,n$,}
\end{equation}
where $G_{\ast}\bvec_{\alpha}$ denotes the push-forward of $\bvec_{\alpha}$ under $G$:
\begin{equation}
\label{eq:pushforward}
G_{\ast}\bvec_{\alpha}
	= \sum_{\gamma=0}^{n} \frac{\pd G_{\gamma}}{\pd x_{\alpha}} \bvec_{\gamma}
	= \sum_{\gamma=0}^{n} \frac{\pd^{2}\psi}{\pd x_{\alpha}\pd x_{\gamma}} \bvec_{\gamma}
	\quad
	\text{$\alpha = 0,1,\dotsc,n$.}
\end{equation}
Indeed, substituting \eqref{eq:pushforward} in \eqref{eq:isometry} yields:
\begin{flalign}
\label{eq:isometry-calc}
\delta(G_{\ast}\bvec_{\alpha},G_{\ast}\bvec_{\beta})
	&= \sum_{\gamma,\kappa=0}^{n} \frac{\pd^{2}\psi}{\pd x_{\alpha} \pd x_{\gamma}} \frac{\pd^{2}\psi}{\pd x_{\beta} \pd x_{\kappa}}
	\delta(\bvec_{\gamma},\bvec_{\kappa})
	= \sum_{\gamma=0}^{n} \frac{\pd^{2}\psi}{\pd x_{\alpha} \pd x_{\gamma}} \frac{\pd^{2}\psi}{\pd x_{\beta} \pd x_{\gamma}}
	\notag\\
	&= \hess(\psi)_{\alpha\beta}^{2}
	= g_{\alpha\beta},
\end{flalign}
so we have established the following result:

\begin{proposition}
\label{prop:isometry}
Let $g$ be a Riemannian metric on the positive open orthant $\orthalt$ of $V=\R^{n+1}$.
If $g = \hess(\psi)^{2}$ for some smooth function $\psi\from\orthalt\to\R$, the derivative map $G=\nabla\psi$ is an isometric injective immersion of $(\orthalt,g)$ in $(V,\delta)$.
\end{proposition}

\smallskip

As it turns out, in the context of \ac{HR} metrics generated by a kernel function $\theta$, $G$ is actually an isometric \emph{embedding} of $(\orthalt,g)$ in $(V,\delta)$ and it can be calculated by a simple, explicit recipe.%
\footnote{Recall that an embedding is an injective immersion that is homeomorphic onto its image \cite{Lee03}.
The existence of isometric embeddings is a consequence of the celebrated Nash\textendash Kuiper embedding theorem;
however, Nash\textendash Kuiper does not provide an explicit construction of such an embedding.}
To do so, let $\phi\from(0,+\infty)\to\R$ be defined as:
\begin{equation}
\label{eq:EC-kernel}
\phi''(x)
	= \sqrt{\theta''(x)},
\end{equation}
and consider the coordinate transformation:
\begin{equation}
\label{eq:EC}
\tag{EC}
x_{\alpha}
	\mapsto \xi_{\alpha}
	= \phi'(x_{\alpha}),
	\quad
	\text{$\alpha=0,1,\dotsc,n$.}
\end{equation}
Letting $\psi(x) = \sum_{\alpha=0}^{n} \phi(x_{\alpha})$, it follows immediately that the derivative map $G=\nabla\psi$ of $\psi$ is closed (i.e. it maps closed sets to closed sets), so the transformation \eqref{eq:EC} is a homeomorphism onto its image, and hence an isometric embedding of $(\orthalt,g)$ in $(\R^{n+1},\delta)$ by Proposition \ref{prop:isometry}.

By this token, the variables $\xi_{\alpha}$ of \eqref{eq:EC} will be referred to as \emph{Euclidean coordinates} for $(\orthalt,g)$.
In these coordinates, the image of $\intstrat$ is the $n$-dimensional hypersurface
\begin{equation}
\label{eq:confold}
\txs
S
	= G(\intstrat)
	= \big\{ \xi\in\R^{n+1}: \insum_{\alpha=0}^{n} (\phi')^{-1}(\xi_{\alpha})=1 \big\},
\end{equation}
so \eqref{eq:ID} can be seen equivalently as a classical mechanical system evolving on $S$.
Specifically, \eqref{eq:EC} yields $\dot\xi_{\alpha} = \phi''(x_{\alpha}) \dot x_{\alpha} = \sqrt{\theta_{\alpha}''} \dot x_{\alpha}$ and $\ddot \xi_{\alpha} = \theta_{\alpha}'''\big/(2\sqrt{\theta_{\alpha}''}) \dot x_{\alpha}^{2} + \sqrt{\theta_{\alpha}''} \ddot x_{\alpha}$, so, after some algebra, we obtain the following expression for the inertial dynamics \eqref{eq:ID} in Euclidean coordinates:
\begin{equation}
\label{eq:ID-E}
\tag{ID-E}
\ddot \xi_{\alpha}
	= \frac{1}{\sqrt{\theta_{\alpha}''}} \left[
	\payv_{\alpha}
	- \insum_{\beta} \left(\Theta''\big/\theta_{\beta}''\right) \payv_{\beta}
	\right]
	+\frac{1}{2} \frac{1}{\sqrt{\theta_{\alpha}''}} \insum_{\beta} \Theta''\theta_{\beta}'''\big/(\theta_{\beta}'')^{2} \dot\xi_{\beta}^{2}
	-\friction \dot \xi_{\alpha}.
\end{equation}
In this way, \eqref{eq:ID-E} represents a classical ``heavy ball'' moving on the hypersurface $S$ under the potential field $\pot$:
the first term of \eqref{eq:ID-E} is simply the projection of the driving force $F = \grad\pot$ on $S$,
the second term is the so-called ``contact force'' which keeps the particle on $S$,
and the third term of \eqref{eq:ID-E} is simply the friction.

\begin{figure}[t]
\centering
\subfigure[The inertial replicator dynamics \eqref{eq:IRD}.]{
\includegraphics[width=.45\textwidth]{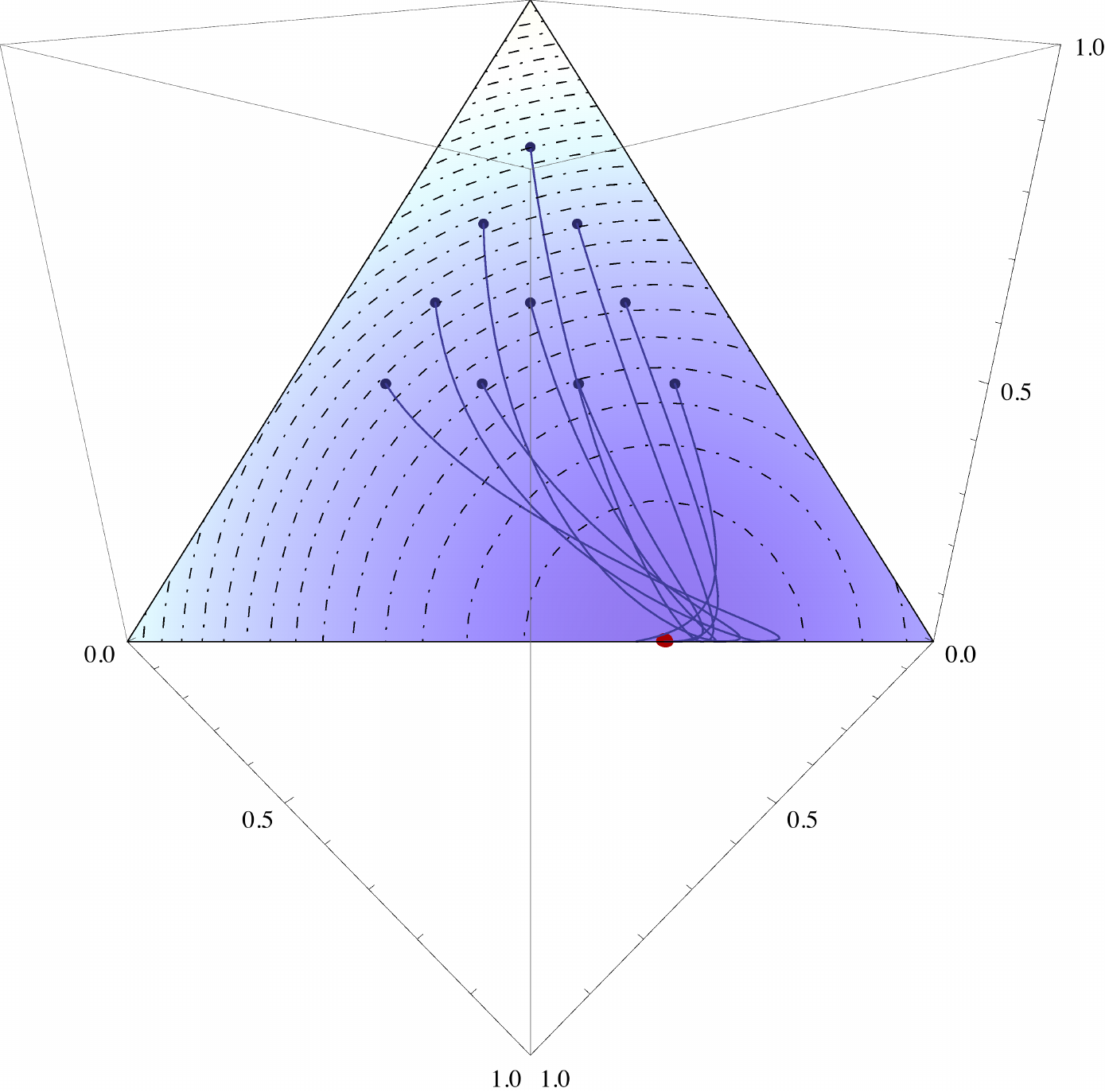}}
\hfill
\subfigure[Euclidean equivalent of \eqref{eq:IRD}.]{
\includegraphics[width=.45\textwidth]{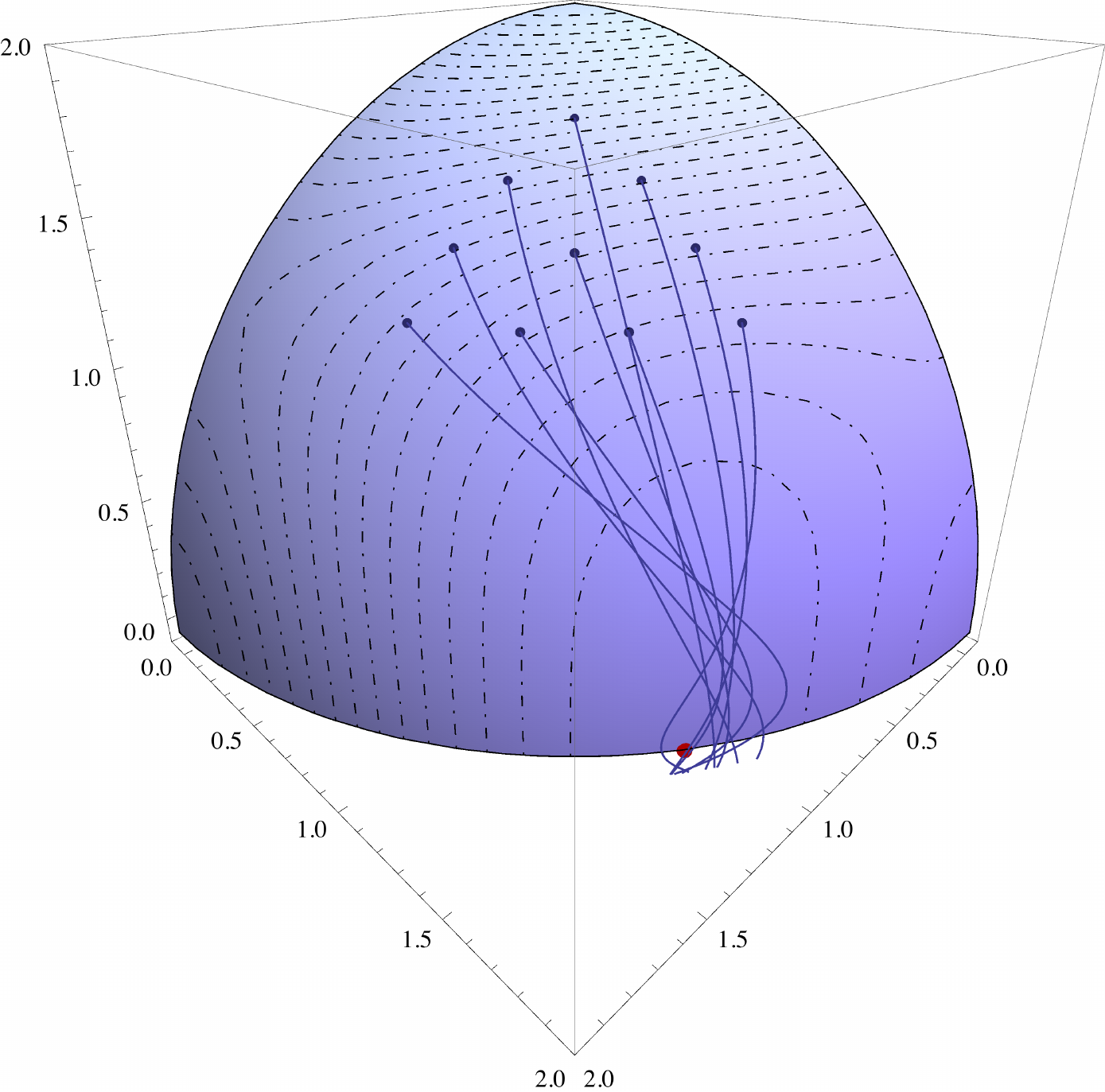}} 
\hfill
\\[1ex]
\subfigure[The inertial log-barrier dynamics \eqref{eq:ILD}.]{
\includegraphics[width=.45\textwidth]{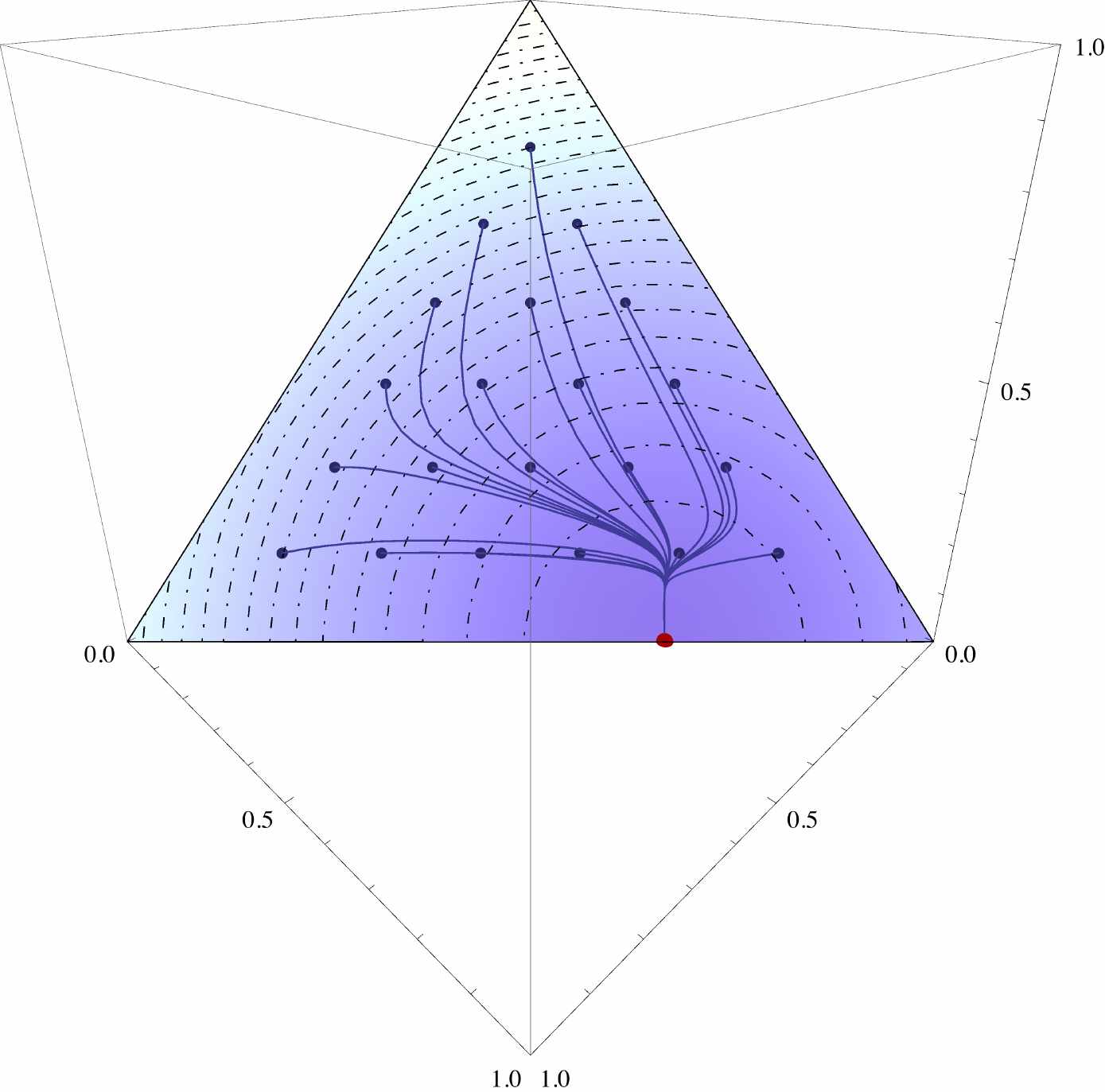}}
\hfill
\subfigure[Euclidean equivalent of \eqref{eq:ILD}.]{
\includegraphics[width=.45\textwidth]{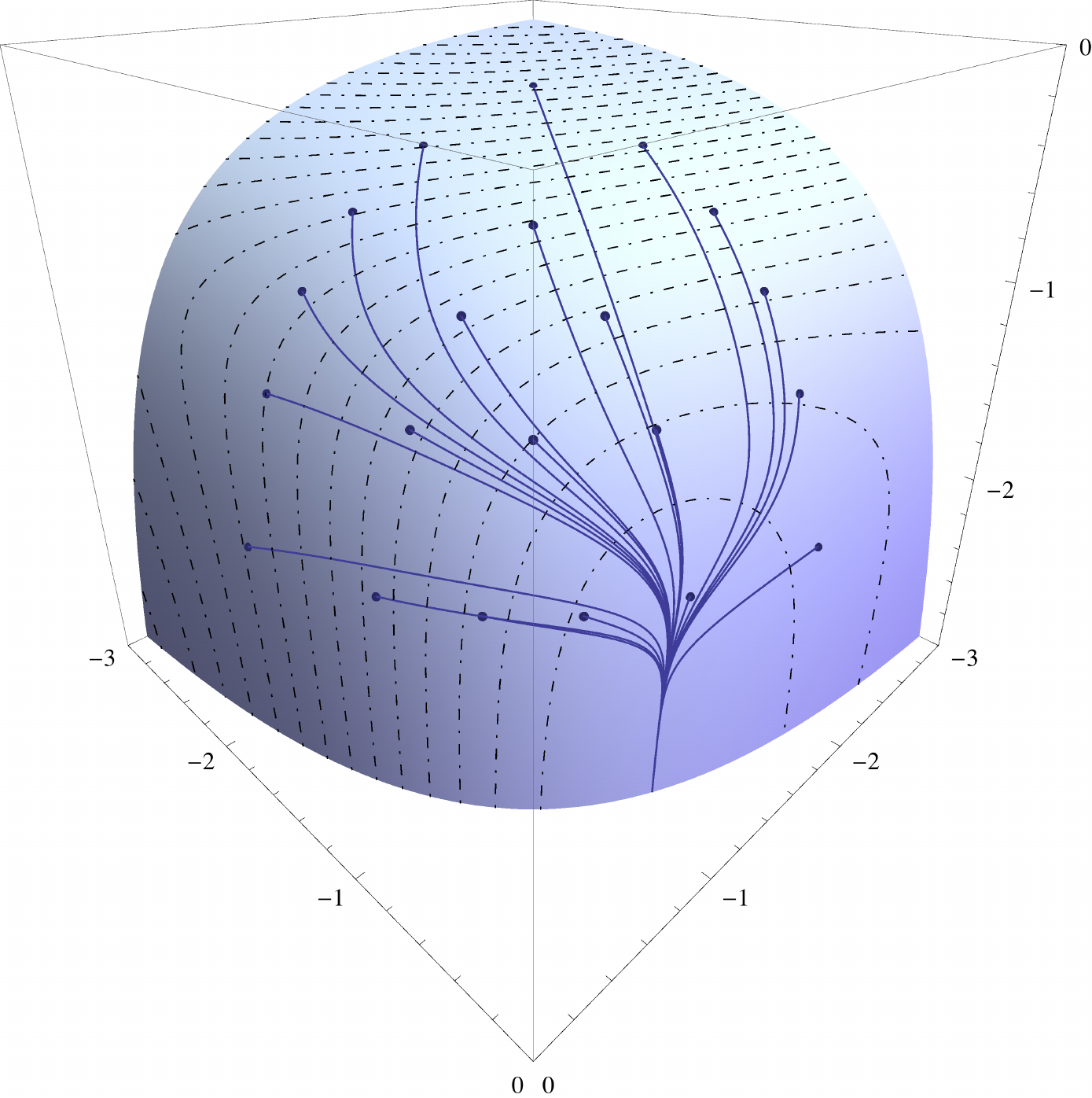}}
\hfill
\\
\caption{%
Asymptotic behavior of the inertial dynamics \eqref{eq:ID} and their Euclidean presentation \eqref{eq:ID-E} for the Shahshahani metric $g^{S}$ \textup(top\textup) and the log-barrier metric $g^{L}$ \textup(bottom\textup) \textendash\ cf.~\eqref{eq:Shah} and \eqref{eq:log} respectively.
The surfaces to the right depict the isometric image \eqref{eq:confold} of the simplex in $\R^{n+1}$ and the contours represent the level sets of the objective function $\pot\from\R^{3}\to\R$ with $\pot(x,y,z) = 1 - (x-2/3)^{2} - (y-1/3)^{2} - z^{2}$.
As can be seen in Figs.~\textup(a\textup) and \textup(b\textup), the inertial replicator dynamics \eqref{eq:IRD} collide with the boundary $\bd(\strat)$ of $\strat$ in finite time and thus fail to maximize $\pot$;
on the other hand, the solution orbits of \eqref{eq:ILD} converge globally to global maximum point of $\pot$.}
\label{fig:wp}
\end{figure}

This reformulation of \eqref{eq:ID} will play an important part in our well-posedness analysis, so we discuss two representative examples:

\begin{example}
\label{ex:EC-S}
In the case of the Shahshahani metric \eqref{eq:Shah}, the transformation \eqref{eq:EC-kernel} gives $\phi''(x) = \sqrt{\theta''(x)} = 1/\sqrt{x}$, so the Euclidean coordinates of the Shahshahani metric are $\xi_{\alpha} = 2\sqrt{x_{\alpha}}$ and $\intstrat$ is isometric to the hypersurface:
\begin{equation}
\label{eq:confold-S}
\txs
S
	= \big\{ \xi\in\R^{n+1} : \xi_{\alpha}>0,\:\insum_{\alpha=0}^{n} \xi_{\alpha}^{2} = 4 \big\},
\end{equation}
which is simply the (open) positive orthant of an $n$-dimensional sphere of radius $2$.%
\footnote{This change of variables was first considered by Akin \cite{Aki79} and is sometimes referred to as \emph{Akin's transformation} \cite{San10}.}
Hence, substituting in \eqref{eq:ID-E}, the Euclidean equivalent of the dynamics \eqref{eq:IRD} will be given by:
\begin{equation}
\label{eq:IRD-E}
\ddot \xi_{\alpha}
	= \frac{1}{2} \xi_{\alpha} \left[\payv_{\alpha} - \frac{1}{4}\insum_{\beta} \xi_{\beta}^{2}\payv_{\beta} \right]
	- \frac{1}{2} \xi_{\alpha} \kin
	- \friction \dot\xi_{\alpha},
\end{equation}
where $K = \tfrac{1}{2}\insum_{\beta} \dot\xi_{\beta}^{2}$ represents the system's kinetic energy.
\end{example}

\begin{example}
\label{ex:EC-L}
In the case of the log-barrier metric \eqref{eq:log}, we have $\phi''(x) = 1/x$, so the metric's Euclidean coordinates are given by the transformation $\xi_{\alpha} = \phi'(x_{\alpha}) = \log x_{\alpha}$.
Under this transformation, $\intstrat$ is mapped isometrically to the hypersurface
\begin{equation}
\label{eq:confold-L}
\txs
S
	= \big\{ \xi\in\R^{n+1}: \xi_{\alpha}<0,\: \insum_{\beta}e^{\xi_{\beta}}=1 \big\},
\end{equation}
which is a closed (non-compact) hypersurface of $\R^{n+1}$.
In these transformed variables, the log-barrier dynamics \eqref{eq:ILD} then become:
\begin{equation}
\label{eq:ILD-E}
\ddot \xi_{\alpha}
	= e^{\xi_{\alpha}} \left[ \payv_{\alpha} - r^{-2}\insum_{\beta} e^{2\xi_{\beta}}\payv_{\beta} \right]
	- r^{-2} e^{\xi_{\alpha}} \insum_{\beta} e^{\xi_{\beta}} \dot \xi_{\beta}^{2}
	- \friction \dot\xi_{\alpha},
\end{equation}
where $r^{2} = \insum_{\beta} x_{\beta}^{2} = \insum_{\beta} e^{2\xi_{\beta}}$.
\end{example}

The above examples highlight an important geometric difference between the dynamics \eqref{eq:IRD} and \eqref{eq:ILD}:
\eqref{eq:IRD} corresponds to a classical particle moving under the influence of a finite force on an open portion of a sphere while \eqref{eq:ILD} corresponds to a classical particle moving under the influence of a finite force on the unbounded hypersurface \eqref{eq:confold-L}.
As a result, physical intuition suggests that trajectories of \eqref{eq:IRD} escape in finite time while trajectories of \eqref{eq:ILD} exist for all time (cf. Fig.~\ref{fig:wp}).

The following theorem (proved in App.~\ref{app:calculations}) shows that this is indeed the case:

\begin{theorem}
\label{thm:wp}
Let $g$ be a \acl{HR} metric on the open orthant $\orthalt = \orthant^{n+1}$ and let $S$ be the image of $\intstrat$ under the Euclidean transformation \eqref{eq:EC}.
Then, the inertial dynamics \eqref{eq:ID} are well-posed on $\intstrat$ if and only if $S$ is a closed hypersurface of $\R^{n+1}$.
\end{theorem}



From a more practical viewpoint, Theorem \ref{thm:wp} allows us to verify that \eqref{eq:ID} is well-posed simply by checking that the Euclidean image $S$ of $\intstrat$ is closed.
More precisely, we have:

\begin{corollary}
\label{cor:wp}
The inertial dynamics \eqref{eq:ID} are well-posed if and only if the kernel $\theta$ of the \ac{HR} structure of $\strat$ satisfies $\int_{0}^{1} \sqrt{\theta''(x)} \dd x = +\infty$.
\end{corollary}

\begin{proof}
Simply note that the image $S=G(\intstrat)$ of $\intstrat$ under \eqref{eq:EC} is bounded (and hence, not closed) if and only if $\int_{0}^{1} \phi'(x) \dd x = \int_{0}^{1} \sqrt{\theta''(x)} \dd x < +\infty$.
\hfill
\end{proof}

In light of the above, we finally obtain:

\begin{corollary}
\label{cor:wp-log}
The inertial log-barrier dynamics \eqref{eq:ILD} are well-posed.
\end{corollary}

%% file: Results.tex

In this section, we investigate the long-term optimization and rationality properties of the inertial dynamics \eqref{eq:ID}.
Specifically, Section \ref{sec:optimization} focuses on the single-agent framework of \eqref{eq:ID} with $\payv = \nabla\pot$ for some smooth (but not necessarily concave) objective function $\pot\from\strat\to\R$;
Section \ref{sec:rationality} then examines the convergence properties of \eqref{eq:ID} in the context of games in normal form (both symmetric and asymmetric).

Since we are interested in the long-term convergence properties of \eqref{eq:ID}, we will assume throughout this section that:
\begin{equation}
\label{eq:WP}
\tag{WP}
\text{The solution orbits $x(t)$ of the inertial dynamics \eqref{eq:ID} exist for all time.}
\end{equation}
Thus, in what follows (and unless explicitly stated otherwise), we will be implicitly assuming that the conditions of Theorem \ref{thm:wp} and Corollary \ref{cor:wp} hold;
as such, the analysis of this section applies to the inertial log-barrier dynamics \eqref{eq:ILD} but not to the inertial replicator system \eqref{eq:IRD} which fails to be well-posed.

\subsection{Convergence and stability properties in constrained optimization}
\label{sec:optimization}

As before, let $\strat \equiv \simplex(n+1)$ be the $n$-dimensional simplex of $V\equiv\R^{n+1}$ and let $\pot\from\strat\to\R$ be a smooth objective function on $\strat$.
Proposition \ref{prop:dissipation} shows that the system's energy is dissipated along \eqref{eq:ID}, so physical intuition suggests that interior trajectories of \eqref{eq:ID} are attracted to (local) maximizers of $\pot$.
We begin by showing that if an orbit spends an arbitrarily long amount of time in the vicinity of some point $\eq\in\strat$, then $\eq$ must be a critical point of $\pot$ restricted to the subface $\eqset$ of $\strat$ that is spanned by $\supp(\eq)$:

\begin{proposition}
\label{prop:stationarity}
Let $x(t)$ be an interior solution of \eqref{eq:ID} that is defined for all $t\geq0$.
Assume further that, for every $\delta>0$ and for every $T>0$, there exists an interval $J$ of length at least $T$ such that $\max_{\alpha}\{|x_{\alpha}(t) - \eq_{\alpha}|\} < \delta$ for all $t\in J$.
Then:
\begin{equation}
\label{eq:stationarity}
\pd_{\alpha} \pot(\eq)
	= \pd_{\beta} \pot(\eq)
\quad
\text{for all $\alpha,\beta\in\supp(\eq)$.}
\end{equation}
\end{proposition}

For the proof of this proposition, we will need the following preparatory lemma:

\begin{lemma}
\label{lem:ineq-diff}
Let $\xi\from [a,b] \to\R$ be a smooth curve in $\R$ such that
\begin{equation}
\label{eq:ineq-diff}
\ddot\xi + \friction \xi \leq -m,
\end{equation}
for some $\friction\geq0,$ $m>0$ and for all $t\in [a,b]$.
Then, for all $t\in [a,b]$, we have:
\begin{equation}
\label{eq:ineq-sol}
\xi(t) \leq \xi(a) +
\begin{cases}
	\friction^{-1}\left(\dot\xi(a) + m\friction^{-1}\right) \left(1 - e^{-\friction (t-a)}\right) - m\friction^{-1} (t-a)
	&\quad \text{if $\friction >0$},\\[5pt]
	\dot\xi(a) (t-a) - \frac{1}{2} m (t-a)^{2}
	&\quad \text{if $\friction=0$}.
\end{cases}
\end{equation}
\end{lemma}

\begin{proof}
The case $\friction=0$ is trivial to dispatch simply by integrating \eqref{eq:ineq-diff} twice.
On the other hand, for $\friction>0$, if we multiply both sides of \eqref{eq:ineq-diff} with $\exp(\friction t)$ and integrate, we get:
\begin{equation}
\label{eq:ineq-vel}
\dot \xi(t)
	\leq \dot\xi(a) e^{-\friction (t-a)} - m\friction^{-1} \left(1- e^{-\friction (t-a)}\right),
\end{equation}
and our assertion follows by integrating a second time.
\hfill
\end{proof}

\begin{proofof}{Proof of Proposition \ref{prop:stationarity}}
Set $\payv_{\beta} = \pd_{\beta}\pot$, $\beta=0,\dotsc,n$, and let $\alpha$ be such that $\payv_{\alpha}(\eq) \leq \payv_{\beta}(\eq)$ for all $\beta\in\supp(\eq)$;
assume further that $\payv_{\alpha}(\eq) \neq \payv_{\gamma}(\eq)$ for some $\gamma\in\supp(\eq)$.
We then have $\payv_{\alpha}(\eq) - \insum_{\beta\in\supp(\eq)} \Theta''(\eq)/\theta_{\beta}''(\eq) \cdot \payv_{\beta}(\eq) < -m' < 0$ for some $m'>0$, and hence, by continuity, there exists some $m>0$ such that
\begin{equation}
\label{eq:payest}
\theta_{\alpha}''(x)^{-1/2}
	\left(
	\payv_{\alpha}(x) - \insum_{\beta} \left(\Theta''(x)\big/\theta_{\beta}''(x)\right) \payv_{\beta}(x)
	\right)
	< m < 0
\end{equation}
for all $x\in \nhd_{\delta} \equiv \{x: \max_{\beta} |x_{\beta} - \eq_{\beta}|<\delta\}$ and for all sufficiently small $\delta>0$ (simply recall that $\lim_{x\to0^{+}}\theta''(x) = +\infty$ and that $\theta_{\alpha}''(\eq) > 0$).

That being so, fix $\delta>0$ as above, and let $M>0$ be such that $x_{\alpha} - \eq_{\alpha} <-\delta$ whenever the Euclidean coordinates $\xi_{\alpha}$ of \eqref{eq:EC} satisfy $\xi_{\alpha}<-M$.
Choose also some sufficiently large $T>0$;
then, by assumption, there exists an interval $J=[a,b]$ with length $b-a\geq T$ and such that $x(t) \in \nhd_{\delta}$ for all $t\in J$.
Since $\lim_{t\to\infty} \dot x_{\alpha}(t) = 0$ by Proposition \ref{prop:stop}, we may also assume that the interval $J=[a,b]$ is such that $\dot \xi_{\alpha}(a)$ is itself sufficiently small (simply note that if $x_{\alpha}$ is bounded away from $0$, $\dot \xi_{\alpha} = \phi''(x_{\alpha})\dot x_{\alpha}$ cannot become arbitrarily large).

In this manner,
the Euclidean presentation \eqref{eq:ID-E} of \eqref{eq:ID} yields
\begin{equation}
\ddot \xi_{\alpha}
	\leq -m
	+\frac{1}{2} \frac{1}{\sqrt{\theta_{\alpha}''}} \insum_{\beta} \Theta''\theta_{\beta}'''\big/(\theta_{\beta}'')^{2} \dot\xi_{\beta}^{2}
	- \friction \dot\xi_{\alpha}
	< -m - \friction \dot\xi_{\alpha}
\quad
\text{for all $t\in J$},
\end{equation}
where the second inequality follows from the regularity assumption $\theta'''(x) < 0$.
However, with $T$ large enough and $\dot \xi_{\alpha}(a)$ small enough, Lemma \ref{lem:ineq-diff} shows that $\xi_{\alpha}(t) < -M$ for large enough $t\in J$, implying that $x(t) \neq \nhd_{\delta}$, a contradiction.
We thus conclude that $\payv_{\alpha}(\eq) = \payv_{\gamma}(\eq)$ for all $\alpha,\gamma\in\supp(\eq)$, as claimed.
\end{proofof}

Proposition \ref{prop:stationarity} shows that if $x(t)$ converges to $\eq\in\strat$, then $\eq$ must be a restricted critical point of $\pot$ in the sense of \eqref{eq:stationarity}.
More generally, the following lemma establishes that any $\omega$-limit of \eqref{eq:ID} has this property:

\begin{lemma}
\label{lem:omega}
Let $\omlim$ be an $\omega$-limit of \eqref{eq:ID} for $\friction>0$, and let $\nhd$ be a neighborhood of $\omlim$ in $\strat$.
Then, for every $T>0$, there exists an interval $J$ of length at least $T$ such that $x(t) \in \nhd$ for all $t\in J$.
\end{lemma}

\begin{proof}
Fix a neighborhood $\nhd$ of $\omlim$ in $\strat$, and let $\nhd_{\delta} = \{x: \max_{\beta} \smallabs{x_{\beta}-x_{\beta}^{\ast}}<\delta\}$ be a $\delta$-neighborhood of $\omlim$ such that $\nhd_{\delta} \cap \strat \subseteq \nhd$.
By assumption, there exists an increasing sequence of times $t_{n}\to\infty$ such that $x(t_{n}) \to \omlim$, so we can take $x(t_{n})\in \nhd_{\delta/2}$ for all $n$.
Moreover, let $t_{n}' = \inf\{t: \text{$t\geq t_{n}$ and $x(t) \notin \nhd_{\delta}$}\}$ be the first exit time of $x(t)$ from $\nhd_{\delta}$ after $t_{n}$, and assume ad absurdum that $t_{n}' - t_{n} < M$ for some $M>0$ and for all $n$.
Then, by descending to a subsequence of $t_{n}$ if necessary, we will have $\abs{x_{\alpha}(t_{n}') - x_{\alpha}(t_{n})} > \delta/2$ for some $\alpha$ and for all $n$.
Hence, by the mean value theorem, there exists $\tau_{n} \in (t_{n},t_{n}')$ such that
\begin{equation}
\abs{\dot x_{\alpha}(\tau_{n})}
	= \frac{\abs{x_{\alpha}(t_{n}') - x_{\alpha}(t_{n})}}{t_{n}' - t_{n}}
	> \frac{\delta}{2M}
	\quad
	\text{for all $n$},
\end{equation}
implying in particular that $\limsup \abs{\dot x_{\alpha}(t)} > \delta/(2M) > 0$ in contradiction to Proposition \ref{prop:stop}.
We thus conclude that the difference $t_{n}' - t_{n}$ is unbounded, i.e. for every $\delta>0$ and for every $T>0$, there exists an interval $J$ of length at least $T$ such that $x(t) \in \nhd_{\delta}$ for all $t\in J$.
\hfill
\end{proof}

Even though the above properties of \eqref{eq:ID} are interesting in themselves (cf. Theorem \ref{thm:folk} for a game-theoretic interpretation), for now they will mostly serve as stepping stones to the following asymptotic convergence result:

\begin{theorem}
\label{thm:opt}
With notation as above, let $\eq\in\strat$ be a local maximizer of $\pot$ such that $(x-\eq)^{\trans} \cdot \hess(\pot(\eq))\cdot (x-\eq) > 0$ for all $x\in\strat$ with $\supp(x)\subseteq \supp(\eq)$ \textendash\ i.e. $\hess(\pot(\eq))$ is positive-definite when restricted to the subface of $\strat$ that is spanned by $\eq$.
If $\friction>0$, then, for every interior solution $x(t)$ of \eqref{eq:ID} that starts close enough to $\eq$ and with sufficiently speed $\norm{\dot x(0)}$, we have $\lim_{t\to\infty} x(t) = \eq$.
\end{theorem}

\begin{proof}
Let $\omlim$ be an $\omega$-limit of $x(t)$.
By Lemma \ref{lem:omega}, $x(t)$ will be spending arbitrarily long time intervals near $\omlim$, so Proposition \ref{prop:stationarity} shows that $\omlim$ satisfies the stationarity condition \eqref{eq:stationarity}, viz. $\pd_{\alpha} \pot(\omlim) = \pd_{\beta} \pot(\omlim) = \payv^{\ast}$ for all $\alpha,\beta\in\supp(\omlim)$.

We will proceed to show that the theorem's assumptions imply that $\eq$ is the unique $\omega$-limit of $x(t)$, i.e. $\lim_{t\to\infty} x(t) = \eq$.
To that end, assume that $x(t)$ starts close enough to $\eq$ and with sufficiently low energy.
Then, Proposition \ref{prop:dissipation} shows that every $\omega$-limit of $x(t)$ must also lie close enough to $\eq$ (simply note that $\pot(x(t))$ can never exceed the initial energy $E(0)$ of $x(t)$);
as a result, the support of any $\omega$-limit of $x(t)$ will contain that of $\eq$.
However, by the theorem's assumptions, the restriction of $\pot$ to the face of $\strat$ spanned by $\eq$ is strongly concave near $\eq$, and since $\omlim$ itself lies close enough to $\eq$, we get:
\begin{equation}
\label{eq:ccv1}
\insum_{\beta=0}^{n}
	\pd_{\beta} \pot(\omlim)\cdot
	(\omlim_{\beta} - \eq_{\beta})
	\leq \pot(\omlim) - \pot(\eq)
	\leq 0,
\end{equation}
with equality if and only if $\omlim = \eq$.
On the other hand, with $\supp(\eq)\subseteq\supp(\omlim)$, we also get:
\begin{flalign}
\label{eq:ccv2}
\sum_{\beta=0}^{n}
	\pd_{\beta} \pot(\omlim)\cdot
	(\eq_{\beta} - \omlim_{\beta})
	&= \sum_{\beta\in\supp(\omlim)}
	\pd_{\beta} \pot(\omlim)\cdot
	(\eq_{\beta} - \omlim_{\beta})
	\notag\\
	&= \payv^{\ast} \sum_{\beta\in\supp(\omlim)} (\eq_{\beta} - \omlim_{\beta})
	= 0,
\end{flalign}
so $\omlim = \eq$, as claimed.
\hfill
\end{proof}

\begin{remark}
Since the total energy $E(t)$ of the system is decreasing, Theorem \ref{thm:opt} implies that $x(t)$ stays close and converges to $\eq$ whenever it starts close to $\eq$ with low energy.
This formulation is almost equivalent to $\eq$ being asymptotically stable in \eqref{eq:ID};
in fact, if $\eq$ is interior, the two statements are indeed equivalent.
For $\eq\in\bd(\strat)$, asymptotic stability is a rather cumbersome notion because the structure of the phase space of the dynamics \eqref{eq:ID} changes at every subface of $\strat$;
in view of this, we opted to stay with the simpler formulation of Theorem \ref{thm:opt} \textendash\ for a related discussion, see \cite[Sec.~5]{LM13}.
\end{remark}

\begin{remark}
\label{rem:isolated}
We should also note here that the non-degeneracy requirement of Theorem \ref{thm:opt} can be relaxed:
for instance, the same proof applies if there is no $x'$ near $\eq$ such that $\pd_{\alpha} \pot(x') = \pd_{\beta} \pot(x')$ for all $\alpha,\beta\in\supp(x')$.
More generally, if $\eqset$ is a convex set of local maximizers of $\pot$ and \eqref{eq:ccv1} holds in a neighborhood of $\eqset$ with equality if and only if $\omlim\in\eqset$, a similar (but more cumbersome) reasoning shows that $x(t)\to\eqset$, i.e. $\eqset$ is locally attracting.
\end{remark}

\begin{remark}
Theorem \ref{thm:opt} is a local convergence result and does not exploit global properties of the objective function (such as concavity) in order to establish global convergence results.
Even though physical intuition suggests that this should be easily possible, the mathematical analysis is quite convoluted due to the boundary behavior of the covariant correction term of \eqref{eq:ID} (the second term of \eqref{eq:ID-E} which acts as a contact force that constrains the trajectories of \eqref{eq:ID-E} to $S$).

The main difficulty is that a Lyapunov-type argument relying on the minimization of the system's total energy $E = \kin - \pot$ does not suffice to exclude convergence to a point $\omlim\in\strat$ that is a local maximizer of $\pot$ on the subface of $\strat$ that is spanned by $\supp(\omlim)$.
In the first-order case, this phenomenon is ruled out by using the Bregman divergence $D_{h}(\eq,x) = h(\eq) - h(x) - h'(x;\eq-x)$ as a global Lyapunov function;
in our context however, the obvious candidate $E_{h} = \kin + D_{h}$ does not satisfy a dissipation principle because of the curvature of $\strat$ under the \ac{HR} metric induced by $h$.
\end{remark}

\subsection{Convergence and stability properties in games}
\label{sec:rationality}

We now return to game theory and examine the convergence and stability properties of \eqref{eq:ID} with respect to Nash equilibria.
To that end, recall first that a strategy profile $\eq=(\eq_{1},\dotsc,\eq_{N})\in\strat$ is called a \emph{Nash equilibrium} if it is stable against unilateral deviations, i.e.
\begin{equation}
\label{eq:Nash}
\pay_{k}(\eq)
	\geq \pay_{k}(x_{k};\eq_{-k})
	\quad
	\text{for all $x_{k}\in\strat_{k}$ and for all $k\in\play$,}
\end{equation}
or, equivalently:
\begin{equation}
\label{eq:Nash-components}
\payv_{k\alpha}(\eq)
	\geq \payv_{k\beta}(\eq)
	\quad
	\text{for all $\alpha\in\supp(\eq_{k})$ and for all $\beta\in\act_{k}$, $k\in\play$.}
\end{equation}
If \eqref{eq:Nash} is strict for all $x_{k}\neq\eq_{k}$, $k\in\play$, we say that $\eq$ a \emph{strict equilibrium};
finally, if \eqref{eq:Nash} holds for all $x_{k}\in\strat_{k}$ such that $\supp(x_{k}) \subseteq \supp(\eq_{k})$, we say that $\eq$ is a \emph{restricted equilibrium} \cite{San10}.

Our first result concerns potential games, viewed here simply as a class of (non-convex) optimization problems defined over products of simplices:

\begin{proposition}
\label{prop:potential}
Let $\game\equiv\game(\play,\act,\pay)$ be a potential game with potential function $\pot$, and let $\eq$ be an isolated maximizer of $\pot$ \textup(and, hence, a strict equilibrium of $\game$\textup).
If $\friction>0$ and $x(t)$ is an interior solution of \eqref{eq:ID} that starts close enough to $\eq$ with sufficiently low initial speed $\norm{\dot x(0)}$, then $x(t)$ stays close to $\eq$ for all $t\geq0$ and $\lim_{t\to\infty} x(t) = \eq$.
\end{proposition}

\begin{proof}
In the presence of a potential function $\pot$ as in \eqref{eq:potential}, the dynamics \eqref{eq:ID} become $D^{2}x/Dt^{2} = \grad\pot - \friction \dot x$ for $x\in\strat\equiv\prod_{k}\simplex(\act_{k})$, so our claim essentially follows as in Theorem \ref{thm:opt}:
Propositions \ref{prop:stop} and \ref{prop:stationarity} extend trivially to the case where $\strat$ is a product of simplices, and, by multilinearity of the game's potential, it follows that there are no other stationary points of \eqref{eq:ID} near a strict equilibrium of $\game$ (cf. Remark \ref{rem:isolated}).
As a result, any trajectory of \eqref{eq:ID} which starts close to a strict equilibrium $\eq$ of $\game$ and always remains in its vicinity will eventually converge to $\eq$;
since trajectories which start near $\eq$ with sufficiently low kinetic energy $\kin(0)$ have this property, our claim follows.
\hfill
\end{proof}

On the other hand, Proposition \ref{prop:potential} does not say much for general, non-potential games.

More generally, if the game does not admit a potential function, the most well-known stability and convergence result is the so-called ``folk theorem'' of evolutionary game theory \cite{HS98,HS03} which states that, under the replicator dynamics \eqref{eq:RD}:
\begin{enumerate}
[I.]
\item
A state is stationary if and only if it is a restricted equilibrium.
\item
If an interior solution orbit converges, its limit is Nash.
\item
If a point is Lyapunov stable, then it is also Nash.
\item
A point is asymptotically stable if and only if it is a strict equilibrium.
\end{enumerate}
In the context of the inertial game dynamics \eqref{eq:ID}, we have:

\begin{theorem}
\label{thm:folk}
Let $\game\equiv\game(\play,\act,\pay)$ be a finite game, let $x(t)$ be a solution orbit of \eqref{eq:ID} that exists for all time, and let $\eq\in\strat$.
Then:
\begin{enumerate}
[\upshape I.]
\item
$x(t) = \eq$ for all $t\geq0$ if and only if $\eq$ is a restricted equilibrium of $\game$ \textup(i.e. $\payv_{k\alpha}(\eq)=\max\{\payv_{k\beta}(\eq): \eq_{k\beta}>0\}$ whenever $\eq_{k\alpha}>0$\textup).
\item
If $x(0)\in\intstrat$ and $\lim_{t\to\infty} x(t) = \eq$, then $\eq$ is a restricted equilibrium of $\game$.
\item
If every neighborhood $\nhd$ of $\eq$ in $\strat$ admits an interior orbit $x_{\nhd}(t)$ such that $x_{\nhd}(t) \in \nhd$ for all $t\geq0$, then $\eq$ is a restricted equilibrium of $\game$.
\item
If $\eq$ is a strict equilibrium of $\game$ and $x(t)$ starts close enough to $\eq$ with sufficiently low speed $\norm{\dot x(0)}$, then $x(t)$ remains close to $\eq$ for all $t\geq0$ and $\lim_{t\to\infty} x(t) = \eq$.
\end{enumerate}
\end{theorem}

\smallskip

\begin{proofof}{Proof of Theorem \ref{thm:folk}}

We begin with the stationarity of restricted Nash equilibria.
Clearly, extending the dynamics \eqref{eq:ID} to $\bd(\strat)$ in the obvious way, it suffices to consider interior stationary equilibria.
Accordingly, if $\eq\in\intstrat$ is Nash, we will have $\payv_{k\alpha}(\eq) = \payv_{k\beta}(\eq)$ for all $\alpha,\beta\in\act_{k}$, and hence also $\payv_{k\alpha}(\eq) = \insum_{\beta}^{k} (\Theta_{k}''/\theta_{k\beta}'') \payv_{k\beta}(\eq)$ for all $\alpha\in\act_{k}$.
Furthermore, with $\theta_{\alpha}''(\eq) > 0$, the velocity-dependent terms of \eqref{eq:ID} will also vanish if $\dot x_{k\alpha}(0) = 0$ for all $\alpha\in\act_{k}$, so the initial conditions $x(0) = \eq$, $\dot x(0) = 0$, imply that $x(t) = \eq$ for all $t\geq0$.
Conversely, if $x(t) = \eq$ for all time, then we also have $\dot x(t) = 0$ for all $t\geq0$, and hence $\payv_{k\alpha}(\eq) = \insum_{\beta}^{k} (\Theta_{k}''/\theta_{k\beta}'') \payv_{k\beta}(\eq)$ for all $\alpha\in\act_{k}$, i.e. $\eq$ is an equilibrium of $\game$.

For Part II of the theorem, note that if an interior trajectory $x(t)$ converges to $\eq\in\strat$, then every neighborhood $\nhd$ of $\eq$ in $\strat$ admits an interior orbit $x_{\nhd}(t)$ such that $x_{\nhd}(t)$ stays in $\nhd$ for all $t\geq0$, so the claim of Part II is subsumed in that of Part III.
To that end, assume ad absurdum that $\eq$ has the property described above without being a restricted equilibrium, i.e. there exists $\alpha\in\supp(\eq_{k})$ with $\payv_{k\alpha}(\eq) < \max_{\beta} \{\payv_{k\beta}(\eq)\}$.
As in the proof of Proposition \ref{prop:stationarity},%
\footnote{Note here that Proposition \ref{prop:stationarity} does not apply directly because the dynamics \eqref{eq:ID} need not be conservative.}
let $\nhd$ be a small enough neighborhood of $\eq$ in $\strat$ such that
\begin{equation}
\label{eq:paydiff}
\theta_{k\alpha}''(x)^{-1/2}
	\left[
	\payv_{k\alpha}(x) - \insum_{\beta}^{k} \left(\Theta_{k}''(x)\big/\theta_{k\beta}''(x)\right) \payv_{k\beta}(x)
	\right]
	< m < 0
\end{equation}
for all $x\in \nhd$.
Then, with $x(t)\in \nhd$ for all $t\geq0$, the Euclidean presentation \eqref{eq:ID-E} of the inertial dynamics \eqref{eq:ID} readily gives
\begin{equation}
\ddot \xi_{k\alpha}
	\leq -m
	+\frac{1}{2} \frac{1}{\sqrt{\theta_{k\alpha}''}} \insum_{\beta}^{k} \Theta_{k}''\theta_{k\beta}'''\big/(\theta_{k\beta}'')^{2} \dot\xi_{k\beta}^{2}
	- \friction \dot\xi_{k\alpha}
	< -m - \friction \dot\xi_{k\alpha}
\quad
\text{for all $t\geq0$},
\end{equation}
so, by Lemma \ref{lem:ineq-diff}, we obtain $\xi_{k\alpha}(t) \to -\infty$ as $t\to\infty$.
However, the definition \eqref{eq:EC} of the Euclidean coordinates $\xi_{k\alpha}$ shows that $x_{k\alpha}(t)\to0$ if $\xi_{k\alpha}(t)\to-\infty$, and since $\eq_{k\alpha}>0$ by assumption, we obtain a contradiction which establishes our original claim.

Finally, for Part IV of the theorem, let $\eq = (\alpha_{1}^{\ast},\dotsc,\alpha_{N}^{\ast})$ be a strict equilibrium of $\game$ (recall that only vertices of $\strat$ can be strict equilibria).
We will show that if $x(t)$ starts at rest ($\dot x(0) = 0$) and with initial Euclidean coordinates $\xi_{k\mu}(0)$, $\mu\in\act_{k}^{\ast} \equiv \act_{k}\exclude\{\alpha_{k}^{\ast}\}$ that are sufficiently close to their lowest possible value $\xi_{k,0}\equiv\inf\{\phi_{k}'(x): x>0\}$,%
\footnote{By the definition of the Euclidean coordinates $\xi_{k\alpha} = \phi_{k}'(x_{k\alpha})$, this condition is equivalent to $x(t)$ starting at a small enough neighborhood of $\eq$.}
then $x(t)\to q$ as $t\to\infty$.
Our proof remains essentially unchanged (albeit more tedious to write down)
if the (Euclidean) norm of the initial velocity $\dot\xi(0)$ of the trajectory is bounded by some sufficiently small constant $\delta>0$, so the theorem follows by recalling that $\smallnorm{\dot x(0)} = \smallnorm{\dot \xi(0)}$.

Indeed, let $\nhd$ be a neighborhood of $\eq$ in $\strat$ such that \eqref{eq:paydiff} holds for all $x\in \nhd$ and for all $\mu\in\act_{k}^{\ast} \equiv \act_{k}\exclude\{\alpha_{k}^{\ast}\}$ substituted in place of $\alpha$.
Moreover, let $\nhd'=G(\nhd)$ be the image of $\nhd$ under the Euclidean embedding $\xi = G(x)$ of Eq.~\eqref{eq:EC}, and let $\tau_{\nhd} = \inf\{t: x(t) \notin \nhd\} = \inf\{t:\xi(t) \notin \nhd'\}$ be the first escape time of $\xi(t) = G(x(t))$ from $\nhd'$.
Assuming $\tau_{\nhd} < +\infty$ (recall that $\xi(t)$ is assumed to exist for all $t\geq0$), we have $x_{k\mu}(\tau_{\nhd}) \geq x_{k\mu}(0)$ and hence $\xi_{k\mu}(\tau_{\nhd}) \geq \xi_{k\mu}(0)$ for some $k\in\play$, $\mu\in\act_{k}^{\ast}$;
consequently, there exists some $\tau' \in (0,\tau_{\nhd}')$ such that $\dot\xi_{k\mu}(\tau') \geq 0$.
By the definition of $\nhd$, we also have $\ddot \xi_{k\mu} + \friction \dot \xi_{k\mu} < -m < 0$ for all $t\in(0,\tau_{\nhd})$, so, with $\dot \xi(0) = 0$, the bound \eqref{eq:ineq-vel} in the proof of Lemma \ref{lem:ineq-diff} readily yields $\dot\xi_{k\mu}(\tau')  < 0$, a contradiction.%
\footnote{One simply needs to consider the escape time $\tilde\tau$ from a larger neighborhood $\tilde \nhd$ of $q$ chosen so that if $\smallabs{\dot\xi_{k\mu}(0)} < \delta$ for some sufficiently small $\delta>0$, then the bound \eqref{eq:velineq} guarantees the existence of a non-positive rate of change $\dot\xi_{k\mu}(\tau_{0})$ for some $\tau_{0}<\tilde\tau$.}
We thus conclude that $\tau_{\nhd} = +\infty$, so we also get $\ddot\xi_{k\mu} + \friction \dot\xi_{k\mu} < -m <0$ for all $k\in\play$, $\mu\in\act_{k}^{\ast}$, and for all $t\geq0$.
Lemma \ref{lem:ineq-diff} then gives $\lim_{t\to\infty} \xi_{k\mu}(t) = -\infty$, i.e. $x(t)\to \eq$, as claimed.
\hfill
\end{proofof}

\smallskip

Theorem \ref{thm:folk} is our main rationality result for asymmetric (multi-population) games, so some remarks are in order:

\begin{remark}
Performing a point-to-point comparison between the first-order ``folk theorem'' of \cite{HS98,HS03} for \eqref{eq:RD} and Theorem \ref{thm:folk} for \eqref{eq:ID}, we may note the following:

\smallskip

Part I of Theorem \ref{thm:folk} is tantamount to the corresponding first-order statement.

\smallskip

Part II differs from the first-order case in that it allows convergence to non-Nash stationary profiles.
For $\friction=0$, the reason for this behavior is that if a trajectory $x(t)$ starts close to a restricted equilibrium $\eq$ with an initial velocity pointing towards $\eq$, then $x(t)$ may escape towards $\eq$ if there is only a vanishingly small force pushing $x(t)$ away from $\eq$.
We have not been able to find such a counterexample for $\friction>0$ and we conjecture that even a small amount of friction prohibits convergence to non-Nash profiles.

\smallskip

Part III only posits the existence of a single interior trajectory that stays close to $\eq$, so it is a less stringent requirement than Lyapunov stability;
on the other hand, and for the same reasons as before, this condition does not suffice to exclude non-Nash stationary points of \eqref{eq:ID}.

\smallskip

Part IV is not exactly the same as the corresponding first-order statement because the notion of asymptotic stability is quite cumbersome in a second-order setting.
Theorem \ref{thm:folk} shows instead that if $x(t)$ starts close to $\eq$ and with sufficiently low speed $\norm{\dot x(0)}$ (or, equivalently, sufficiently low kinetic energy $\kin(0) = \frac{1}{2}\norm{\dot x(0)}^{2}$), then $x(t)$ remains close to $\eq$ and $\lim_{t\to\infty} x(t) = \eq$.
This result continues to hold when restricting \eqref{eq:ID} to any subface $\strat'$ of $\strat$ containing $\eq$, so this can be seen as a form of asymptotic stability for $\eq$.%
\footnote{This could be formalized by considering the phase space obtained by joining the phase space of \eqref{eq:ID} with that of every possible restriction of \eqref{eq:ID} to a subface $\strat'$ of $\strat$, but this is a rather tedious formulation (see also the relevant remark following Theorem \ref{thm:opt}).}
\end{remark}

\begin{remark}
Finally, we note that Theorem \ref{thm:folk} does not require a positive friction coefficient $\friction>0$, in stark contrast to the convergence result of Theorem \ref{thm:opt}.
The reason for this is that convergence to strict equilibria corresponds to the Euclidean trajectories of \eqref{eq:ID-E} escaping towards infinity, so friction is not required to ensure convergence.
As such, Part IV of Theorem \ref{thm:folk} also extends Proposition \ref{prop:potential} to the frictionless case $\temp=0$.
\end{remark}

We close this section with a brief discussion of the rationality properties of \eqref{eq:ID} in the class of symmetric (single-population) games, i.e. $2$-player games where $\act_{1} = \act_{2} = \act$ for some finite set $\act$ and $x_{1} = x_{2}$ \cite{HS98,San10,Wei95}.%
\footnote{In the ``mass-action'' interpretation of evolutionary game theory, this class of games simply corresponds to intra-species interactions in a single-species population \cite{Wei95}.}
In this case, a fundamental equilibrium refinement due to Maynard Smith and Price \cite{MS74,MSP73} is the notion of an \acf{ESS}, i.e. a state that cannot be invaded by a small population of mutants;
formally, we say that $\eq\in\strat\equiv\simplex(\act)$ is \emph{evolutionarily stable} if there exists a neighborhood $\nhd$ of $\eq$ in $\strat$ such that:
\begin{subequations}
\label{eq:ESS}
\begin{flalign}
\label{eq:ESS-Nash}
\pay(\eq,\eq)
	&\geq \pay(x,\eq)
	\quad
	\text{for all $x\in\strat$},
	\\
\label{eq:ESS-invasion}
\pay(\eq,\eq)
	&= \pay(x,\eq)
	\quad
	\text{implies that $\pay(\eq,x) > \pay(x,x)$},
\end{flalign}
\end{subequations}
where $\pay(x,y) = x^{\trans} U y$ is the game's payoff function and $U = (U_{\alpha\beta})_{\alpha,\beta\in\act}$ is the game's payoff matrix.%
\footnote{Intuitively, \eqref{eq:ESS-Nash} implies that $\eq$ is a symmetric Nash equilibrium of the game while \eqref{eq:ESS-invasion} means that $\eq$ performs better against any alternative best reply $x$ than $x$ performs against itself.}
We then have:

\begin{proposition}
\label{prop:ESS}
With notation as above, let $\eq$ be an \ac{ESS} of a symmetric game with symmetric payoff matrix.
Then, provided that $\friction>0$, $\eq$ attracts all interior trajectories of \eqref{eq:ID} that start near $\eq$ and with sufficiently low speed $\norm{\dot x(0)}$.
\end{proposition}

\begin{proof}
Following \cite{Tay79}, recall that $\eq$ is an \ac{ESS} if and only if there exists a neighborhood $\nhd$ of $\eq$ in $\strat$ such that
\begin{equation}
\label{eq:ESS-VI}
\braket{\payv(x)}{x - \eq}
	< 0
	\quad
	\text{for all $x\in\nhd\exclude\{\eq\}$},
\end{equation}
where $\payv_{\alpha}(x) = \pay(\alpha,x)$ denotes the average payoff of the $\alpha$-th strategy in $x\in\strat$.
Since the game's payoff matrix is symmetric, we will also have $\payv(x) = \frac{1}{2}\nabla\pay(x,x)$, so $\eq$ is a local maximizer of $\pay$;
as a result, the conditions of Theorem \ref{thm:opt} are satisfied and our claim follows.
\hfill
\end{proof}

%% file: Discussion.tex

To summarize, the class of inertial game dynamics considered in this paper exhibits some unexpected properties.
First and foremost, in the case of the replicator dynamics, the inertial system \eqref{eq:IRD} does not coincide with the second-order replicator dynamics of exponential learning \eqref{eq:RD-2};
in fact, the dynamics \eqref{eq:IRD} are not even well-posed, so the rationality properties of \eqref{eq:RD-2} do not hold in that case.
On the other hand, by considering a different geometry on the simplex, we obtain a well-posed class of game dynamics with several local convergence and stability properties, some of which do not hold for \eqref{eq:RD-2} (such as the asymptotic stability of \acp{ESS} in symmetric, single-population games).

Having said that, we still have several open questions concerning the dynamics' \emph{global} properties.
From an optimization viewpoint, the main question that remains is whether the dynamics converge globally to a maximum point in the case of concave functions;
in a game-theoretic framework, the main issue is the elimination of stricly dominated strategies (which is true in both \eqref{eq:RD} and \eqref{eq:RD-2}) and, more interestingly, that of \emph{weakly} dominated strategies (which holds under \eqref{eq:RD-2} but not under \eqref{eq:RD}).
A positive answer to these questions (which we expect is the case) would imply that the class of inertial game dynamics combines the advantages of both first- and second-order learning schemes in games, thus collecting a wide array of long-term rationality properties under the same umbrella.

%% file: App-Geometry.tex

In this section, we give a brief overview of the geometric notions used in the main part of the paper following the masterful account of \cite{Lee97}.

Let $W = \R^{n+1}$ and let $W^{\ast}$ be its dual.
A \emph{scalar product} on $W$ is a bilinear pairing $\product{\argdot}{\argdot}\from W\times W\to \R$ such that for all $w,z\in W$:
\begin{enumerate}
\item
$\product{w}{z} = \product{z}{w}$ (\emph{symmetry}).
\item
$\product{w}{w} \geq 0$ with equality if and only if $w=0$ (\emph{positive-definiteness}).
\end{enumerate}
By linearity, if $\{\bvec_{\alpha}\}_{\alpha=0}^{n}$ is a basis for $W$ and $w = \insum_{\alpha=0}^{n} w_{\alpha} \bvec_{\alpha}$, $z = \insum_{\beta=0}^{n} z_{\beta} \bvec_{\beta}$, we have
\begin{equation}
\label{eq:scalar}
\smallproduct{w}{z}
	= \sum_{\alpha,\beta=0}^{n} g_{\alpha\beta} w_{\alpha} z_{\beta},
\end{equation}
where the so-called \emph{metric tensor} $g_{\alpha\beta}$ of the scalar product $\product{\argdot}{\argdot}$ is defined as
\begin{equation}
\label{eq:metric}
g_{\alpha\beta}
	= \smallproduct{\bvec_{\alpha}}{\bvec_{\beta}}.
\end{equation}
Likewise, the \emph{norm} of $w\in W$ is defined as
\begin{equation}
\label{eq:norm}
\txs
\norm{w}
	= \product{w}{w}^{1/2}
	= \left(\insum_{\alpha,\beta=0}^{d} g_{\alpha\beta} w_{\alpha} w_{\beta} \right)^{1/2}.
\end{equation}

Now, if $\open$ is an open set in $W$ and $x\in\open$, the \emph{tangent space} to $\open$ at $x$ is simply the (pointed) vector space $T_{x}\open \equiv \{(x,w):w\in W\} \cong W$ of \emph{tangent vectors} at $x$;
dually, the \emph{cotangent space} to $\open$ at $x$ is the dual space $T_{x}^{\ast} \open \equiv (T_{x}\open)^{\ast} \cong W^{\ast}$ of all linear forms on $T_{x}\open$ (also known as \emph{cotangent vectors}).
Fibering the above constructions over $\open$,
a \emph{vector field} (resp. \emph{differential form}) is then a smooth assignment $x\mapsto w(x) \in T_{x}\open$ (resp. $x\mapsto \omega(x) \in T_{x}^{\ast}\open$),
and the space of vector fields (resp. differential forms) on $\open$ will be denoted by $\sections(\open)$ (resp. $\sections^{\ast}(\open)$).

Given all this, a \emph{Riemannian metric} on $\open$ is a smooth assignment of a scalar product to each tangent space $T_{x}\open$, i.e. a smooth field of (symmetric) positive-definite metric tensors $g_{\alpha\beta}(x)$ prescribing a scalar product between tangent vectors at each $x\in \open$.
Furthermore, if $f\from\open\to\R$ is a smooth function on $\open$, the \emph{differential} of $f$ at $x$ is defined as the (unique) differential form $df(x) \in T_{x}^{\ast} \open$ such that
\begin{equation}
\label{eq:diff-def}
\left.\frac{d}{dt}\right\vert_{t=0} f(\gamma(t))
	= \braket{df(x)}{\dot\gamma(0)}
\end{equation}
for every smooth curve $\gamma\from(-\eps,\eps)\to \open$ with $\gamma(0) = x$.
Dually, given a Riemannian metric on $\open$, the \emph{gradient} of $f$ at $x$ is then defined as the (unique) vector $\grad f(x)\in T_{x}\open$ such that
\begin{equation}
\label{eq:grad-def}
\left.\frac{d}{dt}\right\vert_{t=0} f(\gamma(t))
	= \product{\grad f(x)}{\dot \gamma(0)}
\end{equation}
for all smooth curves $\gamma(t)$ as above.

Combining \eqref{eq:diff-def} and \eqref{eq:grad-def}, we see that $df(x)$ and $\grad f(x)$ satisfy the fundamental duality relation:
\begin{equation}
\label{eq:diff-grad}
\braket{df(x)}{w}
	= \product{\grad f(x)}{w}
	\quad
	\text{for all $w\in T_{x}\open$.}
\end{equation}
Hence, by writing everything out in coordinates and rearranging, we obtain
\begin{equation}
\label{eq:grad}
\left( \grad f(x) \right)_{\alpha}
	= \sum_{\beta=0}^{n} g^{\alpha\beta}(x) \frac{\pd f}{\pd x_{\beta}},
\end{equation}
where
\begin{equation}
\label{eq:ginv}
g^{\alpha\beta}(x)
	= g_{\alpha\beta}^{-1}(x)
\end{equation}
denotes the inverse matrix of the metric tensor $g_{\alpha\beta}(x)$.
For simplicity, we will often write this equation as $\grad f = g^{-1} \nabla f$ where $\nabla f = (\pd_{\alpha}f)_{\alpha=0}^{n}$ denotes the array of partial derivatives of $f$.

In view of the above, differentiating a function $f\in C^{\infty}(\open)$ along a vector field $w\in\sections(\open)$ simply amounts to taking the directional derivative $w(f) \equiv \braket{df}{w} = \product{\grad f}{w}$.
On the other hand, to differentiate a vector field along another, we will need the notion of a (linear) \emph{connection} on $\open$, viz. a map
\begin{equation}
\label{eq:connection}
\del\from \sections(\open) \times \sections(\open) \to \sections(\open)
\end{equation}
written $(w,z)\mapsto \del_{w} z$, and such that:
\begin{subequations}
\label{eq:conn-properties}
\begin{enumerate}
\item
$\del_{f_{1} w_{1} + f_{2} w_{2}} z = f_{1}\del_{w_{1}}z + f_{2}\del_{w_{2}}z$ for all $f_{1},f_{2}\in C^{\infty}(\open)$.
\item
$\del_{w} (az_{1} + bz_{2}) = a\del_{w} z_{1} + b\del_{w} z_{2}$ for all $a,b\in\R$.
\item
$\del_{w} (fz) = f\cdot\del_{w} z + \del_{w} f \cdot z$ for all $f\in C^{\infty}(\open)$, where $\del_{w} f \eqdef w(f) = \braket{df}{w}$.
\end{enumerate}
\end{subequations}
In this way, $\del_{w} z$ generalizes the idea of differentiating $z$ along $w$ and it will be called the \emph{covariant derivative of $z$ in the direction of $w$}.

In the standard frame $\{\bvec_{\alpha}\}_{\alpha=0}^{n}$ of $T\open$, the defining properties of $\del$ give
\begin{equation}
\label{eq:conn-coords}
\del_{w} z
	= \sum_{\alpha,\beta=0}^{n} w_{\alpha} \frac{\pd z_{\beta}}{\pd x_{\alpha}} \bvec_{\beta} + \sum_{\alpha,\beta,\kappa=0}^{n} \Gamma_{\alpha\beta}^{\kappa} w_{\alpha} z_{\beta} \bvec_{\kappa},
\end{equation}
where the \emph{Christoffel symbols} $\Gamma_{\alpha\beta}^{\kappa}\in C^{\infty}(\open)$ of $\del$ in the frame $\{\bvec_{\alpha}\}$ are defined via the equation
\begin{equation}
\label{eq:Christoffel-general}
\del_{\bvec_{\alpha}} \bvec_{\beta}
	= \sum_{\kappa=0}^{n} \Gamma_{\alpha\beta}^{\kappa} \bvec_{\kappa}.
\end{equation}
Clearly, $\del$ is completely specified by its Christoffel symbols, so there is no canonical connection on $\open$;
however, if $\open$ is also endowed with a Riemannian metric $g$, then there exists a \emph{unique} connection which is \emph{symmetric} (i.e. $\Gamma_{\alpha\beta}^{\kappa} = \Gamma_{\beta\alpha}^{\kappa}$)
and \emph{compatible} with $g$ in the sense that:
\begin{equation}
\label{eq:compatibility}
\del_{w} \product{z_{1}}{z_{2}}
	= \product{\del_{w} z_{1}}{z_{2}}
	+ \product{z_{1}}{\del_{w} z_{2}}
	\quad
	\text{for all $w, z_{1}, z_{2} \in \sections(\open)$.}
\end{equation}
This connection is known as the \emph{Levi-Civita connection} on $\open$, and its Christoffel symbols are given in coordinates by
\begin{equation}
\label{eq:Christoffel}
\Gamma_{\alpha\beta}^{\kappa}
	= \frac{1}{2} \sum_{\rho=0}^{n} g^{\kappa\rho}
	\left(
	\frac{\pd g_{\rho\beta}}{\pd x_{\alpha}}
	+ \frac{\pd g_{\rho\alpha}}{\pd x_{\beta}}
	- \frac{\pd g_{\alpha\beta}}{\pd x_{\rho}}
	\right).
\end{equation}

In view of the above, the \emph{covariant derivative} of a vector field $w\in\sections(U)$ along a curve $\gamma(t)$ on $\open$ is defined as:
\begin{equation}
\frac{Dw}{Dt}
	\equiv \del_{\dot\gamma} w
	\equiv \sum_{\alpha,\beta,\kappa=0}^{n}
	\left(
	\dot w_{\kappa}
	+ \Gamma_{\alpha\beta}^{\kappa} w_{\alpha} \dot\gamma_{\beta}
	\right)
	\bvec_{\kappa}.
\end{equation}
Thus, specializing to the case where $w(t)$ is simply the \emph{velocity} $\vel(t) = \dot\gamma(t)$ of $\gamma$, the \emph{acceleration} of $\gamma$ is defined as $\frac{D^{2}\gamma}{Dt^{2}} = \frac{D\vel}{Dt} = \del_{\dot\gamma} \dot\gamma$ or, in components:
\begin{equation}
\label{eq:acceleration}
\frac{D^{2}\gamma_{\kappa}}{Dt^{2}}
	\equiv \ddot \gamma_{\kappa}
	+ \sum_{\alpha,\beta=0}^{n} \Gamma_{\alpha\beta}^{\kappa} \dot\gamma_{\alpha} \dot\gamma_{\beta}.
\end{equation}

The \emph{kinetic energy} of a curve $\gamma(t)$ is defined simply as $\kin = \frac{1}{2} \norm{\dot\gamma}^{2}$;
in view of the metric compatibility condition \eqref{eq:compatibility}, it is then easy to show that
\begin{equation}
\label{eq:kin-diff}
\dot\kin
	= \product{\frac{D^{2}\gamma}{Dt^{2}}}{\dot\gamma},
\end{equation}
so a curve moves at constant speed ($\dot\kin=0$) if and only if it satisfies the geodesic equation $\frac{D^{2}\gamma}{Dt^{2}} = 0$.
On that account, the definition \eqref{eq:acceleration} of a curve's covariant acceleration is simply a consequence of the fundamental requirement that ``curves with zero acceleration move at constant speed'' (by contrast, note that $\ddot\gamma=0$ does not necessarily imply $\dot\kin=0$, so $\ddot\gamma$ cannot act as a covariant measure of acceleration).

%% file: App-Calculations.tex

In this section, we provide some calculations and proofs that would have otherwise disrupted the flow of the paper.

\subsection{Calculation of the Christoffel symbols}

We begin with a matrix inversion formula that is required for our geometric calculations:

\begin{lemma}
\label{lem:inversion}
Let $A_{\mu\nu} = q_{\mu} \delta_{\mu\nu} + q_{0}$ for some $q_{0}, q_{1},\dotsc,q_{n}>0$.
Then, the inverse matrix $A^{\mu\nu}$ of $A_{\mu\nu}$ is
\begin{equation}
\label{eq:inverse}
A^{\mu\nu}
	= \frac{\delta_{\mu\nu}}{q_{\mu}} - \frac{Q}{q_{\mu}q_{\nu}},
\end{equation}
where $Q$ denotes the harmonic aggregate $Q^{-1}\defeq \sum_{\alpha=0}^{n} q_{\alpha}^{-1}$.
\end{lemma}

\begin{proof}
By a straightforward verification, we have:
\begin{flalign}
\sum_{\nu=1}^{n} A_{\mu\nu} A^{\nu\rho}
	&= \insum_{\nu=1}^{n} (q_{\mu} \delta_{\mu\nu} + q_{0}) (\delta_{\nu\rho}/q_{\nu} - Q/(q_{\nu}q_{\rho}))
	\notag\\
	&=\sum_{\nu=1}^{n}
	\left(
	q_{\mu} \delta_{\mu\nu} \delta_{\nu\rho} /q_{\nu}
	+ q_{0} \delta_{\nu\rho}/q_{\nu}
	- q_{\mu} Q \delta_{\mu\nu}/(q_{\nu} q_{\rho})
	- q_{0} Q /(q_{\nu} q_{\rho})
	\right)
	\notag\\
	&=\txs \delta_{\mu\rho} + q_{0}q_{\rho}^{-1} - Q q_{\rho}^{-1} - q_{0} Q q_{\rho}^{-1} \insum_{\nu} q_{\nu}^{-1}
	=\delta_{\mu\rho},
\end{flalign}
as claimed.
\hfill
\end{proof}

With this inversion formula at hand, the inverse matrix $\tilde g^{\mu\nu}$ of the metric tensor $\tilde g_{\mu\nu}$ of $g$ in the coordinates \eqref{eq:pi0} will be given by \eqref{eq:metric-inverse}, viz. $\tilde g^{\mu\nu} = \big[ \delta_{\mu\nu} - \Theta''/\theta_{\nu}'' \big] / \theta_{\mu}''$.
Thus, the Christoffel symbols $\tilde\Gamma_{\mu\nu}^{\kappa}$ of $\tilde g$ in the same coordinate chart can be calculated by the expression $\tilde\Gamma_{\mu\nu}^{\kappa} = \insum_{\rho} \tilde g^{\kappa\rho} \tilde\Gamma_{\rho\mu\nu}$ where, in view of \eqref{eq:Christoffel}, the \emph{Christoffel symbols of the first kind} $\tilde\Gamma_{\rho\mu\nu}$ are defined as:
\begin{equation}
\label{eq:Christoffel2}
\tilde\Gamma_{\rho\mu\nu}
	= \frac{1}{2}\left(
	\frac{\pd \tilde g_{\rho\mu}}{\pd w_{\nu}}
	+ \frac{\pd \tilde g_{\rho\nu}}{\pd w_{\mu}}
	- \frac{\pd \tilde g_{\mu\nu}}{\pd w_{\rho}}
	\right).
\end{equation}

Note now that \eqref{eq:metric-reduced} implies that $\tilde g_{\mu\nu} = \frac{\pd^{2} \tilde h}{\pd x_{\mu} \pd x_{\nu}}$ where $\tilde h = h\circ \iota_{0}\from \open \to \R$ is the pull-back of $h$ to $\open$ via $\iota_{0}$.
By the equality of mixed partials, we then obtain:
\begin{equation}
\label{eq:Christoffel2-w}
\tilde\Gamma_{\rho\mu\nu}
	= \frac{1}{2}\frac{\pd^{3} \tilde h}{\pd w_{\rho} \pd w_{\mu} \pd w_{\nu}}
	= \frac{1}{2}\left(\theta_{\rho}''' \delta_{\rho\mu\nu} - \theta_{0}'''\right),
\end{equation}
where $\delta_{\rho\mu\nu} = \delta_{\rho\mu} \delta_{\mu\nu}$ denotes the triagonal Kronecker symbol ($\delta_{\rho\mu\nu} = 1$ if $\rho = \mu = \nu$ and $0$ otherwise) and  $\theta_{\beta}'''$, $\beta=0,1,\dotsc,n$, is shorthand for $\theta_{\beta}'''(x) = \theta'''(x_{\beta})$.
Accordingly, combining \eqref{eq:Christoffel2-w} and \eqref{eq:metric-inverse}, we finally obtain:
\begin{flalign}
\label{eq:Christoffel-w}
\tilde\Gamma_{\mu\nu}^{\kappa}
	&= \sum_{\rho=1}^{n} \tilde g^{\kappa\rho} \tilde\Gamma_{\rho\mu\nu}
	= \frac{1}{2} \sum_{\rho}
	\left(
	\frac{\delta_{\kappa\rho}}{\theta_{\rho}''} - \frac{\Theta''}{\theta_{\rho}''\theta_{k}''}
	\right)
	(\theta_{\rho}''' \delta_{\rho\mu\nu} - \theta_{0}''')
	\notag\\
	&= \frac{1}{2}
	\left[
	\delta_{\kappa\mu\nu} \frac{\theta_{\kappa}'''}{\theta_{\kappa}''}
	-\frac{\Theta'' \theta_{\mu}'''}{\theta_{\kappa}''\theta_{\mu}''} \delta_{\mu\nu}
	-\frac{\theta_{0}'''}{\theta_{\kappa}''}
	+\frac{\Theta'' \theta_{0}'''}{\theta_{\kappa}''}
	\left(\frac{1}{\Theta''} - \frac{1}{\theta_{0}''}\right)
	\right]
	\notag\\
	&=\frac{1}{2}
	\left[
	\delta_{\kappa\mu\nu} \frac{\theta_{\kappa}'''}{\theta_{\kappa}''}
	-\frac{\Theta'' \theta_{\mu}'''}{\theta_{\kappa}''\theta_{\mu}''} \delta_{\mu\nu}
	-\frac{\theta_{0}'''\Theta''}{\theta_{0}''\theta_{\kappa}''}
	\right],
\end{flalign}
where we used the fact that $\sum_{\rho=1}^{n} 1/\theta_{\rho}'' = 1/\Theta'' - 1/\theta_{0}''$ in the second line.
Consequently, we obtain the following expression for the covariant acceleration \eqref{eq:acceleration} of a curve $x(t)$ on $\open$:
\begin{flalign}
\label{eq:acceleration-w}
\frac{D^{2} x_{\kappa}}{Dt^{2}}
	&=\ddot x_{\kappa}
	+ \frac{1}{2} \sum_{\mu,\nu=1}^{n}
	\left[
	\delta_{\kappa\mu\nu} \frac{\theta_{\kappa}'''}{\theta_{\kappa}''}
	-\frac{\Theta'' \theta_{\mu}'''}{\theta_{\kappa}''\theta_{\mu}''} \delta_{\mu\nu}
	-\frac{\theta_{0}'''\Theta''}{\theta_{0}''\theta_{\kappa}''}
	\right]
	\dot x_{\mu} \dot x_{\nu}
	\notag\\
	&= \ddot x_{\kappa}
	+ \frac{1}{2} \frac{\theta_{\kappa}'''}{\theta_{\kappa}''} \dot x_{\kappa}^{2}
	- \frac{1}{2} \frac{\Theta''}{\theta_{\kappa}''}
	\left[
	\sum_{\nu=1}^{n} \frac{\theta_{\nu}'''}{\theta_{\nu}''} \dot x_{\nu}^{2}
	+ \frac{\theta_{0}'''}{\theta_{0}''} \left(\sum_{\nu=1}^{n} \dot x_{\nu}\right)^{2}
	\right],
\end{flalign}
which is simply \eqref{eq:acceleration-coords}.

\subsection{The well-posedness dichotomy}

In this section, we prove our geometric characterization for the well-posedness of \eqref{eq:ID}:

\begin{proofof}{Proof of Theorem \ref{thm:wp}}
As indicated by our discussion on the inertial systems \eqref{eq:IRD} and \eqref{eq:ILD}, we will prove Theorem \ref{thm:wp} for the equivalent Euclidean dynamics \eqref{eq:ID-E};
also, we will only tackle the frictionless case $\friction=0$, the case $\friction>0$ being entirely similar.
Finally, for notational convenience, the Euclidean inner product will be denoted in what follows by $w\cdot z$ and the corresponding norm by $\abs{\argdot}$.

On account of the above, let $\xi(t)$ be a local solution orbit of \eqref{eq:ID-E} with initial conditions $\xi(0) \equiv \xi_{0}\in S$ and $\dot \xi(0) = \dot\xi_{0}\in T_{\xi_{0}}S$;
existence and uniqueness of $\xi(t)$ follow from the classical Picard\textendash Lindel\"of theorem, so assume ad absurdum that $\xi(t)$ only exists up to some maximal time $T>0$.
Accordingly, let
\begin{subequations}
\begin{flalign}
\label{eq:tangential}
F_{\alpha}
	&=\frac{1}{\sqrt{\theta_{\alpha}''}} \left(
	\payv_{\alpha}
	- \insum_{\beta} \left(\Theta''\big/\theta_{\beta}''\right) \payv_{\beta}
	\right),
\intertext{and}
\label{eq:centripetal}
N_{\alpha}
	&=\frac{\Theta''}{2\sqrt{\theta_{\alpha}''}} \insum_{\beta} \theta_{\beta}'''\big/(\theta_{\beta}'')^{2} \dot\xi_{\beta}^{2},
\end{flalign}
\end{subequations}
denote the tangential and contact force terms of \eqref{eq:ID-E} respectively.
Since $F$ is a weighted difference of bounded quantities, we will have $\abs{F(\xi(t))} \leq \Fsup$ for some $\Fsup>0$;
furthermore, it is easy to verify that $N$ is indeed normal to $S$, so, for all $t<T$, the work of the resultant force $F+N$ along $\xi(t)$ will be:
\begin{equation}
\label{eq:Wineq}
W(t)	= \int_{\xi} (F + N)
	=\int_{0}^{t}  F(\xi(s)) \cdot \dot \xi(s) \dd s
	\leq \Fsup \int_{0}^{t} \smallabs{\dot\xi(s)} \dd s
	\leq \Fsup\cdot\ell(t),
\end{equation}
where $\ell(t) = \int_{0}^{t} \smallabs{\dot\xi(s)} \dd s$ is the (Euclidean) length of $\xi$ up to time $t$.

On the other hand, with $F+N = \ddot\xi$, we will also have
\begin{equation}
W(t) = \int_{0}^{t} \ddot \xi(s) \cdot \dot \xi(s) \dd s = \tfrac{1}{2}\vel^{2}(t) - \tfrac{1}{2}\vel_{0}^{2},
\end{equation}
where $\vel(t) = \smallabs{\dot\xi(t)} = \dot\ell(t)$ is the speed of the trajectory at time $t$ and $\vel_{0} \equiv \smallabs{\dot\xi_{0}}$.
Combining with \eqref{eq:Wineq}, we thus get the differential inequality
\begin{equation}
\label{eq:elldiffineq}
\vel(t) = \dot\ell(t) \leq \sqrt{\vel_{0}^{2} + 2 \Fsup\,\ell(t)},
\end{equation}
which, after separating variables and integrating, gives:
\begin{equation}
\sqrt{\vel_{0}^{2} + 2\Fsup\,\ell(t)} - \vel_{0} \leq \Fsup\, t.
\end{equation}
It thus follows that the speed $\vel(t)$ of the trajectory is bounded by $\smallabs{\dot\xi(t)} = \vel(t) \leq \vel_{0} t + \Fsup t$;
similarly, for the total distance travelled by $\xi(t)$, we get $\ell(t) \leq \vel_{0} t + \tfrac{1}{2} \Fsup t^{2}$, so $\abs{\xi}$ and $\smallabs{\dot\xi}$ are both bounded by some $\ell_{\textup{max}}$ and $\vsup$ respectively for all $t\leq T$.

As a result, for any $s,t \in [0,T)$ with $s<t$, we will also have
\begin{equation}
\abs{\xi(t) - \xi(s)}
	\leq \int_{s}^{t} \smallabs{\dot \xi(\tau)} \dd\tau
	\leq \vsup (t-s),
\end{equation}
so, if $t_{n}\to T$ is Cauchy, the same will for $\xi(t_{n})$ as well;
hence, with $S$ closed, we will also have $\lim_{t\to T}\xi(t) \equiv \xi_{T} \in S$.
With $\dot\xi$ bounded, we then get
\begin{equation}
\smallabs{\dot\xi(t) - \dot\xi(s)}
	\leq \int_{s}^{t} \smallabs{\ddot \xi(\tau)} \dd\tau
	\leq \Fsup (t-s) + \insum_{\beta} \int_{s}^{t} \vert N_{\beta}(\xi(\tau),\dot\xi(\tau))\vert \dd\tau,
\end{equation}
and with $\sup\abs{\xi},\sup \smallabs{\dot\xi} < \infty$, it follows that the components $\vert N_{\beta}\vert$ of the contact force are also bounded:
$x(t) = G^{-1}(\xi(t))$ remains a positive distance away from $\bd(\strat)$ for all $t\leq T$, so the weight coefficients $\theta_{\beta}'''/(\theta_{\beta}'')^{2}$ of the centripetal force $N$ in \eqref{eq:centripetal} are bounded, and the same holds for the velocity components $\dot \xi_{\beta}^{2}$.
We will thus have $\smallabs{\dot\xi(t) - \dot\xi(s)} \leq a (t-s)$ for some $a>0$, so the limit $\lim_{t\to T} \dot\xi(t)$ exists and is finite.
In this way, if we take \eqref{eq:ID-E} with initial conditions $\xi(T) = \xi_{T}$ and $\dot\xi(T) = \lim_{t\to T}\dot \xi(t)$, the Picard\textendash Lindel\"of theorem shows that the original maximal solution $\xi(t)$ may be extended beyond the maximal integration time $T$, a contradiction.

For the converse implication, assume that $S$ is not closed in the ambient space $V\equiv\R^{n+1}$, let $\overline S$ denote its closure, and let $q\in \overline S\exclude{S}$.
Clearly, $\overline S$ is a closed submanifold-with-boundary of $V$ and the metric induced by the inclusion $\overline S \injects V$ on $\overline S$ will agree with the one induced by the inclusion $S\injects V$ on $S$.
With this in mind,
let $\gamma(t)$ be a geodesic of $\overline S$ which starts at $q$ with initial velocity $\vel_{0}$ pointing towards the interior of $S$, and let $T>0$ be sufficiently small so that $\gamma(T)=p\in S^{\circ}$.
Furthermore, let $\vel_{T} = \dot \gamma(T)$ be the velocity with which $\gamma(t)$ reaches $p$;
by the invariance of the geodesic equation with respect to time reflections, this means that the geodesic which starts at $p$ with velocity $-\vel_{T}$ will reach $q$ at finite time $T>0$ with outward-pointing velocity $-\vel_{0}$.
Noting that geodesics on $S$ are simply solutions of \eqref{eq:ID-E} for $\payv\equiv0$ and $\friction=0$, and carrying \eqref{eq:ID-E} back to $\intstrat$ via the isometry \eqref{eq:EC}, we have shown that \eqref{eq:ID} admits a solution which escapes from $\intstrat$ in finite time, i.e. \eqref{eq:ID} is not well-posed if $S$ is not closed.%
\footnote{For general $\payv$ and $\friction>0$, simply let $\gamma(t)$ be a solution of the dynamics $\ddot\xi = F + N +\friction \dot \xi$, i.e. \eqref{eq:ID-E} with $\friction$ replaced by $-\friction$.
The time-reflected variant of this equation is simply \eqref{eq:ID-E}, so the rest of the argument follows in the same way.}
\hfill
\end{proofof}